\documentclass[12pt, reqno]{amsart}
\usepackage{amsmath, amsthm, amssymb}
\usepackage[margin=1.0in]{geometry}
\usepackage{xcolor}
\definecolor{cite}{rgb}{0.30,0.60,1.00}
\definecolor{url}{rgb}{0.00,0.00,0.80}
\definecolor{link}{rgb}{0.40,0.10,0.20}

\usepackage{zref-clever}

\AddToHook{env/theorem/begin}{%
	\zcsetup{countertype={theorem=theorem}}}
\AddToHook{env/claim/begin}{%
	\zcsetup{countertype={theorem=theorem}}}  
\AddToHook{env/proposition/begin}{%
	\zcsetup{countertype={theorem=proposition}}}
\AddToHook{env/lemma/begin}{%
	\zcsetup{countertype={theorem=lemma}}}
\AddToHook{env/conjecture/begin}{%
	\zcsetup{countertype={theorem=theorem}}}  
\AddToHook{env/corollary/begin}{%
	\zcsetup{countertype={theorem=corollary}}}
\AddToHook{env/definition/begin}{%
	\zcsetup{countertype={theorem=definition}}}
\AddToHook{env/remark/begin}{%
	\zcsetup{countertype={theorem=remark}}}
\AddToHook{env/example/begin}{%
	\zcsetup{countertype={theorem=example}}}

\zcsetup{nameinlink=false}

\usepackage[pdfusetitle,colorlinks,linkcolor=link,urlcolor=url,citecolor=cite,pagebackref,breaklinks]{hyperref}
\usepackage{mathdots}
\usepackage{mathtools}

\newtheorem{theorem}{Theorem}[section]

\newtheorem{proposition}[theorem]{Proposition}

\newtheorem{corollary}[theorem]{Corollary}
\theoremstyle{definition}

\newtheorem{remark}[theorem]{Remark}

\newcommand{\Cref}[1]{\zcref[S]{#1}}

\newcommand{\zIntegers}{\mathbb{Z}}
\newcommand{\cComplex}{\mathbb{C}}
\newcommand{\coefficientsField}{\mathbb{K}}
\newcommand{\multiplicativegroup}[1]{#1^{\times}}
\newcommand{\ImageOf}{\mathrm{Im}}
\newcommand{\Hom}{\mathrm{Hom}}
\newcommand{\Isom}{\mathrm{Isom}}
\newcommand{\EndomorphismRing}{\operatorname{End}}
\newcommand{\idmap}{\mathrm{id}}

\newcommand{\sizeof}[1]{\left|#1\right|}
\newcommand{\transpose}[1]{\, {}^{t}#1}
\newcommand{\zetaIntegral}{\mathrm{Z}}
\newcommand{\naturalZetaIntegral}{\mathrm{Z}^{\natural}}

\newcommand{\innerproduct}[2]{\left\langle #1,#2\right\rangle}

\newcommand{\fieldCharacter}{\psi}

\newcommand{\IdentityMatrix}[1]{I_{#1}}
\newcommand{\diag}{\mathrm{diag}}
\newcommand{\trace}{\operatorname{tr}}
\newcommand{\GL}{\mathrm{GL}}
\newcommand{\Mat}[2]{\operatorname{Mat}_{#1 \times #2}}

\newcommand{\finiteField}{\mathbb{F}}
\newcommand{\squareMatrix}{\operatorname{Mat}}
\newcommand{\GaussSum}[2]{\mathcal{G}\left(#1, #2\right)}

\newcommand{\fourierTransform}[1]{\mathcal{F}_{#1}}
\newcommand{\GaussSumScalar}[2]{\mathrm{G}\left(#1, #2\right)}
\newcommand{\GaussSumScalarFunction}[2]{\mathrm{G}_{#1, #2}}

\newcommand{\rank}{\mathrm{rk}}
\newcommand{\srank}{\mathrm{srk}}

\newcommand{\Kovacs}{Kov\'acs}
\newcommand{\Schwartz}{\mathcal{S}}
\newcommand{\evaluation}{\mathrm{eval}}
\newcommand{\regular}{\mathrm{reg}}
\newcommand{\singular}{\mathrm{sng}}
\newcommand{\algebraIsomorphism}{\mathfrak{c}}
\newcommand{\abs}[1]{\left|#1\right|}
\newcommand{\differential}{\mathrm{d}}
\newcommand{\mdifferential}{\differential^{\times}}

\newcommand{\kostkaFoulkes}{\mathrm{K}}
\newcommand{\coinvKostkaFoulkes}{\tilde{\mathrm{K}}^{\ast}}

\newcommand{\pHJacquetModule}[2]{J^H_{+}\left(#1, #2\right)}
\newcommand{\pnHJacquetModule}[2]{J^H_{\pm}\left(#1, #2\right)}
\newcommand{\nHJacquetModule}[2]{J^H_{-}\left(#1, #2\right)}
\newcommand{\ProjectionOperator}{\operatorname{pr}}

\hypersetup{pdfauthor={Elad Zelingher},
	pdfsubject={Number theory, Representation theory},
	pdfkeywords={Gauss sums, Godement--Jacquet, Fourier transform, Exponential sums}}

\title[Godement--Jacquet functional equation]{On the Godement--Jacquet functional equation for finite matrix monoids}

\author{Elad Zelingher}
\address{Department of Mathematics, University of California, Los Angeles, CA, 90095 USA}
\email{eladz@math.ucla.edu}


\begin{document}

\begin{abstract}
	In previous work with Harman and Snowden~\cite{HarmanSnowdenZelingher2025}, we found explicit formulas for units of rank ideals of the matrix monoid algebra $\coefficientsField\left[\squareMatrix_n\left(\finiteField_q\right)\right]$, where $\coefficientsField$ is an algebraically closed field with characteristic not dividing $q$. In this work, we use these explicit formulas to establish a regularized Godement--Jacquet functional equation valid for irreducible representations of $\GL_n\left(\finiteField_q\right)$ and of the algebra $\coefficientsField\left[\squareMatrix_n\left(\finiteField_q\right)\right]$. Using this functional equation we give an explicit expression for primitive central idempotents corresponding to irreducible representations of $\coefficientsField\left[\squareMatrix_n\left(\finiteField_q\right)\right]$, where the characteristic of $\coefficientsField$ does not divide $\sizeof{\GL_n\left(\finiteField_q\right)}$. Interestingly, the coefficients in this formula are closely related to degenerate non-abelian Gauss sums.
\end{abstract}
\maketitle

\tableofcontents

\section{Introduction}

$L$-functions are important objects in number theory. In the 1970s, Godement--Jacquet~\cite{GodementJacquet1972} constructed a zeta integral for the standard $L$-function of an irreducible cuspidal automorphic representation of $\GL_n$. Their work is a natural generalization of the construction given by Tate~\cite{Tate1967} in his doctoral thesis for $L$-functions of Hecke characters. Like Tate, Godement--Jacquet showed that their zeta integral converges in a right half-plane, and that it admits a meromorphic continuation to the entire plane. They showed that for good data this zeta integral is (almost) the standard $L$-function attached to the automorphic representation in question. They used the Poisson summation formula to establish the functional equation for this standard $L$-function.

This construction, like the construction in Tate's thesis, relies on local zeta integrals and their corresponding functional equations. Let $F$ be a local field and let $\pi$ be an admissible irreducible representation of $\GL_n\left(F\right)$. The local Godement--Jacquet zeta integral is defined as
\begin{displaymath}
Z\left(s,v,v^{\vee}, f\right) = \int_{\GL_n\left(F\right)} \innerproduct{\pi\left(g\right) v}{v^{\vee}}f\left(g\right) \abs{\det g}^{s + \frac{n-1}{2}} \mdifferential g.
\end{displaymath}
Here $v \in \pi$, $v^{\vee} \in \pi^{\vee}$, $s \in \cComplex$ and $f \colon \squareMatrix_n\left(F\right) \to \cComplex$ is a Schwartz function. This integral converges for $s$ in a right half-plane (depending only on $\pi$) and has a meromorphic continuation to the entire plane, which we continue to denote by the same symbol. Godement--Jacquet proved that this zeta integral satisfies a functional equation, which we explain now. They showed that for every non-trivial additive character $\fieldCharacter \colon F \to \multiplicativegroup{\cComplex}$, there exists a function $\gamma\left(s, \pi, \fieldCharacter\right)$ (a rational function in $q^{-s}$ in the non-archimedean case) such that for every $v$, $v^{\vee}$, and $f$, we have the equality of meromorphic functions
\begin{displaymath}
Z\left(1-s, v^{\vee}, v, \fourierTransform{\fieldCharacter} f\right) = \gamma\left(s, \pi, \fieldCharacter\right) Z\left(s, v, v^{\vee}, f\right),
\end{displaymath}
where the left-hand side is the zeta function associated to the contragredient representation $\pi^{\vee}$ and $\fourierTransform{\fieldCharacter}f$ is the Fourier transform of $f$, normalized so that $\fourierTransform{\fieldCharacter^{-1}} \fourierTransform{\fieldCharacter} f = f$.

One can define a finite field version of the zeta integral of Godement--Jacquet by considering irreducible representations of $\GL_n\left(\finiteField\right)$, where $\finiteField$ is a finite field with $q$ elements. Macdonald did this in~\cite{Macdonald80} and established the functional equation for the finite field case. Roditty-Gershon also established this functional equation in her master's thesis~\cite{RodittyGershon2010}. Recently, this was also done in the modular case by Kurinczuk--Matringe--S{\'e}cherre~\cite{KurinczukMatringeSecherre2026}. In~\cite{Macdonald80, RodittyGershon2010}, the functional equation is only valid for irreducible representations that do not contain the trivial character $1 \colon \multiplicativegroup{\finiteField} \to \multiplicativegroup{\cComplex}$ in their cuspidal support. As shown by Macdonald, the gamma factor in this case is given by the Kondo Gauss sum attached to $\pi$. Recently, Kurinczuk--Matringe--S{\'e}cherre showed in~\cite[Proposition 2.5]{KurinczukMatringeSecherre2026} that the Godement--Jacquet functional equation holds for all irreducible representations if one restricts to functions supported on $\GL_n\left(\finiteField\right)$.

In this work we provide a regularized Godement--Jacquet functional equation that works for every irreducible representation $\pi$ of $\GL_n\left(\finiteField\right)$ and for every function $f \colon \squareMatrix_n\left(\finiteField\right) \to \coefficientsField$ (\Cref{thm:godement-jacquet-for-arbitrary-functions}). A hint in this direction was given in Roditty-Gershon's thesis~\cite[Theorem 3.1]{RodittyGershon2010} (see also~\cite[Lemma 4.6]{KurinczukMatringeSecherre2026}), in which she established a regularized functional equation for the finite field version of Tate's zeta integral that is valid for all characters $\chi \colon \multiplicativegroup{\finiteField} \to \multiplicativegroup{\cComplex}$ by adding an error term to the right-hand side, without changing the left-hand side.

To achieve this, we construct a function $f^{\regular} \colon \squareMatrix_n\left(\finiteField\right) \to \coefficientsField$ such that its Fourier transform agrees with that of $f$ on $\GL_n\left(\finiteField\right)$ and such that $f^{\regular}$ vanishes outside of $\GL_n\left(\finiteField\right)$. This is done using our recent results with Harman and Snowden~\cite{HarmanSnowdenZelingher2025} in which we found an explicit formula for the unit of the ideal spanned by singular elements $\coefficientsField\left[\squareMatrix_n\left(\finiteField\right)\right]_{\singular} \subset \coefficientsField\left[\squareMatrix_n\left(\finiteField\right)\right]$. The following theorem follows from \Cref{prop:properties-of-regular-function} and \Cref{thm:godement-jacquet-for-arbitrary-functions}.
\begin{theorem}
	Let $f \colon \squareMatrix_n\left(\finiteField\right) \to \coefficientsField$. Define $f^{\regular} \colon \squareMatrix_n\left(\finiteField\right) \to \coefficientsField$ by
	\begin{displaymath}
		f^{\regular}\left(X\right) = \mu\left(n\right)^{-1} \sum_{\substack{m \in \squareMatrix_n\left(\finiteField\right)\\
					m \text{ semi-idempotent}}} \mu\left(\srank m\right) f\left(Xm\right).
	\end{displaymath}
	Then $f^{\regular}$ is supported on $\GL_n\left(\finiteField\right)$ and $\fourierTransform{\fieldCharacter} f^{\regular}$ coincides with $\fourierTransform{\fieldCharacter} f$ on $\GL_n\left(\finiteField\right)$. Moreover, if the characteristic of $\coefficientsField$ does not divide $\sizeof{\GL_n\left(\finiteField\right)}$ then $f^{\regular}$ is the unique function $\squareMatrix_n\left(\finiteField\right) \to \coefficientsField$ supported on $\GL_n\left(\finiteField\right)$ whose Fourier transform coincides with that of $f$ on $\GL_n\left(\finiteField\right)$.
	\end{theorem}

	As a corollary, by invoking the Godement--Jacquet functional equation for functions supported on $\GL_n\left(\finiteField\right)$ (see~\cite[Proposition 2.5]{KurinczukMatringeSecherre2026} and \Cref{thm:functional-equation-first-version}), we obtain the following regularized Godement--Jacquet functional equation.
	\begin{theorem}
		For every irreducible representation $\pi$ of $\GL_n\left(\finiteField\right)$ and every $f \colon \squareMatrix_n\left(\finiteField\right) \to \coefficientsField$, the following equation holds: 
	\begin{equation*}
		\begin{split}
			\sum_{g \in \GL_n\left(\finiteField\right)} \fourierTransform{\fieldCharacter}f\left(g^{-1}\right) \pi\left(g\right) =& \GaussSumScalar{\pi}{\fieldCharacter} \sum_{g \in \GL_n\left(\finiteField\right)} f\left(g\right) \pi\left(g\right) \\
			& + \mu\left(n\right)^{-1} \GaussSumScalar{\pi}{\fieldCharacter} \sum_{g \in \GL_n\left(\finiteField\right)} \sum_{\substack{m \in \squareMatrix_n\left(\finiteField\right)\\
					m \text{ singular}\\
					m \text{ semi-idempotent}}} \mu\left(\srank m\right) f\left(gm\right) \pi\left(g\right).
		\end{split}
	\end{equation*}
	Here $\GaussSumScalar{\pi}{\fieldCharacter}$ is Kondo's Gauss sum (see \Cref{subsec:kondo-gauss-sum}).
\end{theorem}

Next, we are able to extend this result to a Godement--Jacquet functional equation valid for all irreducible representations of the monoid algebra $\coefficientsField\left[\squareMatrix_n\left(\finiteField\right)\right]$. See \Cref{subsec:functional-equation-for-monoid-algebras} and \Cref{thm:godement-jacquet-for-monoid-algebras}.

As an application of these functional equations, we find an explicit formula for primitive central idempotents corresponding to irreducible representations of $\coefficientsField\left[\squareMatrix_n\left(\finiteField\right)\right]$ where $\coefficientsField$ has characteristic not dividing $\sizeof{\GL_n\left(\finiteField\right)}$ (\Cref{thm:ln-sigma-idempotent-in-terms-of-parabolic-induction-and-gauss-sums}). To our knowledge, such an explicit formula was not previously known in the literature.
\begin{theorem}	
	Assume that the characteristic of $\coefficientsField$ does not divide $\sizeof{\GL_n\left(\finiteField\right)}$. Let $\sigma$ be an irreducible representation of $\GL_r\left(\finiteField\right)$ for $0 \le r \le n$. Then the primitive central idempotent $e_{L_n\left(\sigma\right)} \in \coefficientsField\left[\squareMatrix_n\left(\finiteField\right)\right]$ corresponding to the simple module $L_n\left(\sigma\right)$ of $\coefficientsField\left[\squareMatrix_n\left(\finiteField\right)\right]$ is given by
	\begin{displaymath}
	e_{L_n\left(\sigma\right)} = \frac{\dim L_n\left(\sigma\right)}{\sizeof{\GL_n\left(\finiteField\right)}} \sizeof{\GL_{n-r}\left(\finiteField\right)} \cdot \GaussSumScalar{\sigma^{\vee}}{\fieldCharacter}^{-1} \sum_{A \in \squareMatrix_n\left(\finiteField\right)} \left(\GaussSumScalarFunction{\sigma^{\vee}}{\fieldCharacter} \odot \delta_{0_{n-r}}\right) \left(A\right) \left[A\right],
	\end{displaymath}
	where
	\begin{itemize}
		\item $L_n\left(\sigma\right)$ is the irreducible representation of $\coefficientsField\left[\squareMatrix_n\left(\finiteField\right)\right]$ induced from $\sigma$~\cite[Corollary 5.46]{Steinberg2016}, see \Cref{subsec:representations-of-m-n-f-q}.
		\item $\GaussSumScalarFunction{\sigma^{\vee}}{\fieldCharacter}$ is the degenerate non-abelian Gauss sum function associated to $\sigma^{\vee}$ and $\fieldCharacter$~\cite{GorodetskyZelingher2026}, see \Cref{subsubsec:degenerate-non-abelian-gauss-sums, subsubsec:explicit-formulas-for-gauss-sums}.
		\item $\odot$ is parabolic induction, see \Cref{subsubsec:rewrite-using-parabolic-induction}.
		\item $\delta_{0_{n-r}}$ is the indicator function of $0_{n-r} \in \squareMatrix_{n-r}\left(\finiteField\right)$.
	\end{itemize}
\end{theorem}
One interesting feature of this formula is that the coefficients are given by degenerate non-abelian Gauss sums, objects arising from harmonic analysis. In an upcoming work with Gorodetsky~\cite{GorodetskyZelingher2026} we compute these degenerate non-abelian Gauss sums explicitly. This allows one to write down explicit formulas for the central primitive idempotent $e_{L_n\left(\sigma\right)}$. We review the main results of~\cite{GorodetskyZelingher2026} in \Cref{subsubsec:explicit-formulas-for-gauss-sums}.

As a corollary of these results, we obtain \Cref{cor:trace-godement-jacquet-for-m-n-fq}, which allows us to compute the pairing of the primitive central idempotent $e_{L_n\left(\sigma\right)}$ with an arbitrary function $f \colon \squareMatrix_n\left(\finiteField\right) \to \coefficientsField$, in terms of the pairing of the character of $L_n\left(\sigma\right)$ and the Fourier transform of $f$. In fact, this identity was the motivation for this present work. In more detail, in the upcoming work~\cite{HarmanSnowdenZelingher2026}, together with Harman and Snowden, we study interesting tensor categories called \emph{Jacobi categories} (see also~\cite[Section 7.2]{Snowden2026} and~\cite[Section 6.5]{EtingofSnowden2026}). The identity in \Cref{cor:trace-godement-jacquet-for-m-n-fq} allows us to compute the formal dimensions of the simple objects in these Jacobi categories.

\subsection*{Acknowledgements}

I would like to thank Andrew Snowden for introducing me to the representation theory of the monoid algebra $\coefficientsField\left[\squareMatrix_n\left(\finiteField_q\right)\right]$ and for many discussions on this topic. I would like to thank Stephen DeBacker and Ofir Gorodetsky for their interest in this project and their encouragement.

\subsection*{AI Acknowledgements}

ChatGPT Pro models 5.5 and 5.6 were used for proofreading the manuscript. In addition, the author benefited from discussions with these models. The models suggested an early version of the statement in \Cref{thm:godement-jacquet-for-monoid-algebras} and a proof for it. The models also pointed out the formula $f_{\sigma} = \diag\left(e_{\sigma},\dots,e_{\sigma}\right)$ in \Cref{subsubsec:central-primitive-idempotents} and suggested the idea for the proof of \Cref{thm:inverse-image-of-Psi}. The manuscript was entirely written by the author and he takes full responsibility for its content.

\section{Godement--Jacquet functional equation}

Let $\finiteField$ be a finite field with $q$ elements. Let $\coefficientsField$ be an algebraically closed field with characteristic not dividing $q$. We fix a square root of $q$ in $\coefficientsField$, denoted by $\sqrt{q}$. For every $m \in \zIntegers$ we define $q^{\frac{m}{2}} = \sqrt{q}^m$ and $\sqrt{q^m} = \sqrt{q}^m$.

\subsection{Monoid algebra}

Let $\coefficientsField\left[\squareMatrix_n\left(\finiteField\right)\right]$ be the monoid algebra of $\squareMatrix_n\left(\finiteField\right)$. It consists of $\coefficientsField$-linear combinations of elements $\left[X\right]$ where $X \in \squareMatrix_n\left(\finiteField\right)$. Multiplication is defined on basis elements by the rule $\left[X_1\right] \left[X_2\right] = \left[X_1 X_2\right]$ for every $X_1, X_2 \in \squareMatrix_n\left(\finiteField\right)$, and extended bilinearly to all elements.

We have a left $\GL_n\left(\finiteField\right) \times \GL_n\left(\finiteField\right)$ action on $\coefficientsField\left[\squareMatrix_n\left(\finiteField\right)\right]$ defined by the rule $\lambda\left(g,h\right)\left[X\right] = \left[h X g^{-1}\right]$ for every $X \in \squareMatrix_n\left(\finiteField\right)$ and $g,h \in \GL_n\left(\finiteField\right)$.

Let
\begin{displaymath}
 \Schwartz\left(\squareMatrix_n\left(\finiteField\right)\right) = \left\{ f \colon \squareMatrix_n\left(\finiteField\right) \to \coefficientsField \right\}.
 \end{displaymath}
   We have a left $\GL_n\left(\finiteField\right) \times \GL_n\left(\finiteField\right)$ action on $\Schwartz\left(\squareMatrix_n\left(\finiteField\right)\right)$ via
   \begin{displaymath}
   \left(\rho\left(g,h\right) f\right)\left(X\right) = f\left(h^{-1} X g\right),
   \end{displaymath}
where $g,h \in \GL_n\left(\finiteField\right)$.

We may identify $\coefficientsField\left[\squareMatrix_n\left(\finiteField\right)\right]$ with $\Schwartz\left(\squareMatrix_n\left(\finiteField\right)\right)$ via the map sending an element $f \in \Schwartz\left(\squareMatrix_n\left(\finiteField\right)\right)$ to the linear combination
\begin{displaymath}
\algebraIsomorphism\left(f\right) = \sum_{Y \in \squareMatrix_n\left(\finiteField\right)} f\left(Y\right) \left[Y\right].
\end{displaymath}
This map is an isomorphism of left $\coefficientsField\left[\GL_n\left(\finiteField\right) \times \GL_n\left(\finiteField\right)\right]$-modules.

\subsubsection{Duality}
We may regard each of $\coefficientsField\left[\squareMatrix_n\left(\finiteField\right)\right]$ and $\Schwartz\left(\squareMatrix_n\left(\finiteField\right)\right)$ as $\coefficientsField\left[\squareMatrix_n\left(\finiteField\right) \times \squareMatrix_n\left(\finiteField\right)\right]$-modules as follows:
\begin{itemize}
	\item On $\coefficientsField\left[\squareMatrix_n\left(\finiteField\right)\right]$ the action is defined on the generators by
	\begin{displaymath}
		\lambda'\left(\left[A\right],\left[B\right]\right)\left[X\right] = \left[B \cdot X \cdot \transpose{A}\right],
	\end{displaymath}
	for every $X,A,B \in \squareMatrix_n\left(\finiteField\right)$.
	\item On $\Schwartz\left(\squareMatrix_n\left(\finiteField\right)\right)$ the action is defined via
	\begin{displaymath}
		\left(\rho'\left(\left[A\right],\left[B\right]\right)f\right)\left(X\right) = f\left(\transpose{B} \cdot X \cdot A\right)
	\end{displaymath}
	for every $f \in \Schwartz\left(\squareMatrix_n\left(\finiteField\right)\right)$ and $X,A,B \in \squareMatrix_n\left(\finiteField\right)$.
\end{itemize}
Notice that the restriction of these actions to $\coefficientsField\left[\GL_n\left(\finiteField\right) \times \GL_n\left(\finiteField\right)\right]$ coincides with the actions we defined above, up to a twist with the outer automorphism of $\GL_n\left(\finiteField\right)$ given by $g \mapsto g^{\iota} \coloneq \transpose{g}^{-1}$. Precisely, for $A, B \in \GL_n\left(\finiteField\right)$ we have $\rho'\left(\left[A\right],\left[B\right]\right) = \rho\left(A,B^{\iota}\right)$ and $\lambda'\left(\left[A\right],\left[B\right]\right) = \lambda\left(A^{\iota},B\right)$.

Given a $\coefficientsField\left[\squareMatrix_n\left(\finiteField\right)\right]$-module $M$, we define a dual $\coefficientsField\left[\squareMatrix_n\left(\finiteField\right)\right]$-module structure on $\Hom_{\coefficientsField}\left(M, \coefficientsField\right)$ by setting \begin{displaymath}
	\innerproduct{m}{\left[X\right] \cdot m^{\ast}} = \innerproduct{\left[\transpose{X}\right] \cdot m}{m^{\ast}},
\end{displaymath} for every $m \in M$ and $m^{\ast} \in \Hom_{\coefficientsField}\left(M, \coefficientsField\right)$ and $X \in \squareMatrix_n\left(\finiteField\right)$. We let $\transpose{M}$ denote this module. Notice that the usual dual of the restriction of $M$ to $\GL_n\left(\finiteField\right)$ can be realized as the restriction of $\transpose{M}$ to $\GL_n\left(\finiteField\right)$ composed with the outer automorphism $g \mapsto g^{\iota}$ of $\GL_n\left(\finiteField\right)$.

We have that the $\coefficientsField\left[\squareMatrix_n\left(\finiteField\right) \times \squareMatrix_n\left(\finiteField\right)\right] = \coefficientsField\left[\squareMatrix_n\left(\finiteField\right)\right] \otimes_{\coefficientsField} \coefficientsField\left[\squareMatrix_n\left(\finiteField\right)\right]$-modules $\Schwartz\left(\squareMatrix_n\left(\finiteField\right)\right)$ and $\transpose{\coefficientsField\left[\squareMatrix_n\left(\finiteField\right)\right]}$ are isomorphic. This can be easily seen using the bilinear form $B \colon \coefficientsField\left[\squareMatrix_n\left(\finiteField\right)\right] \otimes_{\coefficientsField} \Schwartz\left(\squareMatrix_n\left(\finiteField\right)\right) \to \coefficientsField$, given by
\begin{displaymath}
	B\left(\algebraIsomorphism\left(f_1\right), f_2\right) = \sum_{X \in \squareMatrix_n\left(\finiteField\right)} f_1\left(X\right) f_2\left(X\right),
\end{displaymath}
where $f_1, f_2 \in \Schwartz\left(\squareMatrix_n\left(\finiteField\right)\right)$. One easily checks that for every $A,B \in \squareMatrix_n\left(\finiteField\right)$ we have
\begin{displaymath}
	B\left(\lambda'\left(\left[B\right], \left[A\right]\right)\algebraIsomorphism\left(f_1\right), f_2\right) = B\left(\algebraIsomorphism\left(f_1\right), \rho'\left(\left[\transpose{B}\right], \left[\transpose{A}\right]\right)f_2\right).
\end{displaymath}

\subsection{Fourier transform}
Let $\fieldCharacter \colon \finiteField \to \multiplicativegroup{\coefficientsField}$ be a non-trivial additive character.

Given an element $f \in \Schwartz\left(\squareMatrix_n\left(\finiteField\right)\right)$ we define its \emph{Fourier transform} $\fourierTransform{\fieldCharacter} f \in \Schwartz\left(\squareMatrix_n\left(\finiteField\right)\right)$ by the formula
\begin{displaymath}
\fourierTransform{\fieldCharacter}{f}\left(X\right) = q^{-\frac{n^2}{2}} \sum_{Y \in \squareMatrix_n\left(\finiteField\right)} f\left(Y\right) \fieldCharacter\left(\trace\left(XY\right)\right),
\end{displaymath}
where $X \in \squareMatrix_n\left(\finiteField\right)$.

The Fourier inversion formula takes the form
\begin{displaymath}
\fourierTransform{\fieldCharacter^{-1}} \fourierTransform{\fieldCharacter} f\left(X\right) = f\left(X\right),
\end{displaymath}
or alternatively
\begin{displaymath}
\fourierTransform{\fieldCharacter} \fourierTransform{\fieldCharacter} f\left(X\right) = f\left(-X\right).
\end{displaymath}

Using the monoid algebra $\coefficientsField\left[\squareMatrix_n\left(\finiteField\right)\right]$, we have the following alternative description of the Fourier transform. Let $\evaluation_{\fieldCharacter} \colon \coefficientsField\left[\squareMatrix_n\left(\finiteField\right)\right] \to \coefficientsField$ be the linear map defined on basis elements by $\evaluation_{\fieldCharacter}\left(\left[X\right]\right) = q^{-\frac{n^2}{2}} \fieldCharacter\left(\trace X\right)$. Then
\begin{displaymath}
\fourierTransform{\fieldCharacter}f\left(X\right) = \evaluation_{\fieldCharacter}\left(\left[X\right] \cdot \algebraIsomorphism\left(f\right)\right).
\end{displaymath}

\subsection{Kondo Gauss sum}\label{subsec:kondo-gauss-sum}
We have that the map $\GL_n\left(\finiteField\right) \to \coefficientsField$ given by $g \mapsto \fieldCharacter\left(\trace g^{-1}\right)$ is a class function. Thus if $\pi$ is a representation of $\GL_n\left(\finiteField\right)$, the non-abelian Gauss sum $\GaussSum{\pi}{\fieldCharacter}$, defined as the operator
\begin{displaymath}
\GaussSum{\pi}{\fieldCharacter} = q^{-\frac{n^2}{2}} \sum_{g \in \GL_n\left(\finiteField\right)} \fieldCharacter\left(\trace g^{-1}\right)\pi\left(g\right),
\end{displaymath}
 lies in $\Hom_{\GL_n\left(\finiteField\right)}\left(\pi, \pi\right)$. Suppose that $\pi$ is irreducible. Then by Schur's lemma the latter Hom-space is one-dimensional and therefore there exists a scalar $\GaussSumScalar{\pi}{\fieldCharacter} \in \coefficientsField$ such that $\GaussSum{\pi}{\fieldCharacter} = \GaussSumScalar{\pi}{\fieldCharacter} \cdot \idmap_{\pi}$.

For an irreducible representation $\pi$ of $\GL_n\left(\finiteField\right)$, the scalar $\GaussSumScalar{\pi}{\fieldCharacter}$ was explicitly computed by Kondo in~\cite{Kondo1963} for $\coefficientsField$ of characteristic zero. Kurinczuk--Matringe--S{\'e}cherre~\cite[Proposition 6.1]{KurinczukMatringeSecherre2026} extended his result to the positive characteristic setting.

\begin{remark}\label{rem:kondo-gauss-sum-is-non-zero}
	The scalar $\GaussSumScalar{\pi}{\fieldCharacter}$ is never zero. This follows for example from~\cite[Proposition 4.9]{KurinczukMatringeSecherre2026}.
\end{remark}

\subsection{Functional equation I: functions supported on $\GL_n\left(\finiteField\right)$}
We establish a first version of the Godement--Jacquet functional equation.

Let $\Schwartz\left(\GL_n\left(\finiteField\right)\right)$ be the subspace of $\Schwartz\left(\squareMatrix_n\left(\finiteField\right)\right)$ consisting of all functions $f \colon \squareMatrix_n\left(\finiteField\right) \to \coefficientsField$ such that $f\left(X\right) = 0$ for every $X \notin \GL_n\left(\finiteField\right)$.

Let
\begin{displaymath}
e^{\vee}_{\fieldCharacter} = q^{-\frac{n^2}{2}} \sum_{g \in \GL_n\left(\finiteField\right)} \fieldCharacter\left(\trace g^{-1}\right) \left[g\right] \in \coefficientsField\left[\GL_n\left(\finiteField\right)\right].
\end{displaymath}
Notice that since $g \mapsto \fieldCharacter\left(\trace g^{-1}\right)$ is a class function, we have that $e_{\fieldCharacter}^{\vee}$ lies in the center of $\coefficientsField\left[\GL_n\left(\finiteField\right)\right]$.

\begin{proposition}\label{prop:identity-in-the-group-algebra}
	For every $f \in \Schwartz\left(\GL_n\left(\finiteField\right)\right)$ we have the identity
	\begin{displaymath}
	\algebraIsomorphism\left(f\right) \cdot e^{\vee}_{\fieldCharacter} = e^{\vee}_{\fieldCharacter} \cdot \algebraIsomorphism\left(f\right) = \sum_{g \in \GL_n\left(\finiteField\right)} \fourierTransform{\fieldCharacter}f\left(g^{-1}\right) \left[g\right].
	\end{displaymath}
\end{proposition}
\begin{proof}
	Since $e_{\fieldCharacter}^{\vee}$ is central, it suffices to prove the second equality. We have
	\begin{displaymath}
	e^{\vee}_{\fieldCharacter} \cdot \algebraIsomorphism\left(f\right) = \sum_{g \in \GL_n\left(\finiteField\right)} \sum_{Y \in \GL_n\left(\finiteField\right)} q^{-\frac{n^2}{2}} \fieldCharacter\left(\trace\left(g^{-1}\right)\right) f\left(Y\right) \left[gY\right].
	\end{displaymath}
	Changing variables $g \mapsto g Y^{-1}$, we get
	\begin{displaymath}
	e^{\vee}_{\fieldCharacter} \cdot \algebraIsomorphism\left(f\right) = \sum_{g \in \GL_n\left(\finiteField\right)} \sum_{Y \in \GL_n\left(\finiteField\right)} q^{-\frac{n^2}{2}} \fieldCharacter\left(\trace\left(Y g^{-1}\right)\right) f\left(Y\right) \left[g\right].
	\end{displaymath}
	Since $f$ is only supported on elements of $\GL_n\left(\finiteField\right)$, this equals
	\begin{displaymath}
	e^{\vee}_{\fieldCharacter} \cdot \algebraIsomorphism\left(f\right) = \sum_{g \in \GL_n\left(\finiteField\right)} \sum_{Y \in \squareMatrix_n\left(\finiteField\right)} q^{-\frac{n^2}{2}} \fieldCharacter\left(\trace\left(Y g^{-1}\right)\right) f\left(Y\right) \left[g\right] = \sum_{g \in \GL_n\left(\finiteField\right)} \fourierTransform{\fieldCharacter}f\left(g^{-1}\right) \left[g\right].
	\end{displaymath}
\end{proof}

As a corollary, we obtain the following first version of the functional equation.
\begin{theorem}\label{thm:functional-equation-first-version}
	Let $\pi$ be an irreducible representation of $\GL_n\left(\finiteField\right)$. Then for every $f \in \Schwartz\left(\GL_n\left(\finiteField\right)\right)$ we have the equality
	\begin{displaymath}
	\sum_{g \in \GL_n\left(\finiteField\right)} \fourierTransform{\fieldCharacter} f\left(g^{-1}\right) \pi\left(g\right) = \GaussSumScalar{\pi}{\fieldCharacter} \sum_{g \in \GL_n\left(\finiteField\right)} f\left(g\right) \pi\left(g\right).
	\end{displaymath}
\end{theorem}
\begin{proof}
	Consider the homomorphism of algebras $\coefficientsField\left[\GL_n\left(\finiteField\right)\right] \to \EndomorphismRing_{\coefficientsField}\left(\pi\right)$ sending $\left[g\right] \mapsto \pi\left(g\right)$ for every $g \in \GL_n\left(\finiteField\right)$. Applying this homomorphism to the identity in \Cref{prop:identity-in-the-group-algebra} yields the result.
\end{proof}

\begin{remark}
	This form of the functional equation was also recently established in~\cite[Proposition 2.5]{KurinczukMatringeSecherre2026}.
\end{remark}

\begin{corollary}\label{cor:fourier-transform-is-zero-on-gln-implies-f-is-zero}
	Suppose that the characteristic of $\coefficientsField$ does not divide $\sizeof{\GL_n\left(\finiteField\right)}$. Let $f \in \Schwartz\left(\GL_n\left(\finiteField\right)\right)$ be such that $\fourierTransform{\fieldCharacter}f\left(g\right) = 0$ for every $g \in \GL_n\left(\finiteField\right)$. Then $f = 0$.
\end{corollary}
\begin{proof}
	By \Cref{thm:functional-equation-first-version} and \Cref{rem:kondo-gauss-sum-is-non-zero}, we have that for every irreducible representation $\pi$ of $\GL_n\left(\finiteField\right)$,
	\begin{displaymath}
		\sum_{g \in \GL_n\left(\finiteField\right)}f\left(g\right) \pi\left(g\right) = 0.
	\end{displaymath}
	Since the characteristic of $\coefficientsField$ does not divide $\sizeof{\GL_n\left(\finiteField\right)}$, we have that every finite-dimensional representation of $\GL_n\left(\finiteField\right)$ is semisimple. It follows that one can replace $\pi$ in the equality above with any finite-dimensional representation (not necessarily irreducible) of $\GL_n\left(\finiteField\right)$. Choosing $\pi\left(g\right) = \lambda\left(\IdentityMatrix{n},g\right)$ and applying it to $\left[\IdentityMatrix{n}\right]$, it follows that
	\begin{displaymath}
		\sum_{g \in \GL_n\left(\finiteField\right)} f\left(g\right) \left[g\right] = 0,
	\end{displaymath}
	which implies that $f = 0$.
\end{proof}

\subsection{Functional equation II: general functions}
Our next goal is to improve \Cref{thm:functional-equation-first-version} to obtain a functional equation that holds for every $f \in \Schwartz\left(\squareMatrix_n\left(\finiteField\right)\right)$. One can try to achieve this by replacing $f$ with a function $f_0$ supported on $\GL_n\left(\finiteField\right)$ whose restriction to $\GL_n\left(\finiteField\right)$ agrees with $f$. However, using this naive approach it is hard to keep track of the Fourier transform $\fourierTransform{\fieldCharacter} f_0$ and its relation to the original Fourier transform $\fourierTransform{\fieldCharacter} f$.

Instead, given a function $f \in \Schwartz\left(\squareMatrix_n\left(\finiteField\right)\right)$ we will show how to construct an explicit function $f^{\regular} \in \Schwartz\left(\GL_n\left(\finiteField\right)\right)$ such that $\fourierTransform{\fieldCharacter} f^{\regular}\restriction_{\GL_n\left(\finiteField\right)} = \fourierTransform{\fieldCharacter}f\restriction_{\GL_n\left(\finiteField\right)}$. Then by \Cref{thm:functional-equation-first-version} we will obtain the functional equation
\begin{displaymath}
\sum_{g \in \GL_n\left(\finiteField\right)} \fourierTransform{\fieldCharacter} f\left(g^{-1}\right) \pi\left(g\right) = \GaussSumScalar{\pi}{\fieldCharacter} \sum_{g \in \GL_n\left(\finiteField\right)} f^{\regular}\left(g\right) \pi\left(g\right).
\end{displaymath}

\subsubsection{Monoid of singular matrices}

Let $\coefficientsField\left[\squareMatrix_n\left(\finiteField\right)\right]_{\singular} \subset \coefficientsField\left[\squareMatrix_n\left(\finiteField\right)\right]$ be the $\coefficientsField$-linear span of $\left[X\right]$, where $X$ goes over all the singular elements of $\squareMatrix_n\left(\finiteField\right)$, i.e., over all $X \in \squareMatrix_n\left(\finiteField\right)$ such that $X \notin \GL_n\left(\finiteField\right)$. Then $\coefficientsField\left[\squareMatrix_n\left(\finiteField\right)\right]_{\singular}$ is an ideal of $\coefficientsField\left[\squareMatrix_n\left(\finiteField\right)\right]$. \Kovacs{}~\cite{Kovacs1992} proved that $\coefficientsField\left[\squareMatrix_n\left(\finiteField\right)\right]_{\singular}$ has a multiplicative identity element $\eta_{\singular}$. The element $\eta_{\singular}$ is central in $\coefficientsField\left[\squareMatrix_n\left(\finiteField\right)\right]$. In~\cite[Theorem 4.1]{HarmanSnowdenZelingher2025}, together with Harman and Snowden, we proved the following explicit formula for $\eta_{\singular}$.
\begin{displaymath}
\eta_{\singular} = -\mu\left(n\right)^{-1} \sum_{\substack{m \in \squareMatrix_n\left(\finiteField\right) \\
m \text{ singular}\\
m \text{ semi-idempotent}}} \mu\left(\srank m\right) \left[m\right].
\end{displaymath}
Here $\mu\left(j\right) = \left(-1\right)^j q^{\binom{j}{2}}$ for a non-negative integer $j$. A matrix $m$ is called \emph{semi-idempotent} if there exists $N \ge 0$ such that $m^{N+1} = m^N$. In this case, we set $m^{\infty} = m^N$ and $\srank m = \rank m^{\infty}$ and call $\srank m$ the \emph{stable rank} of $m$. Equivalently, $m$ is semi-idempotent if it can be realized in some basis as a block diagonal matrix consisting of an identity matrix and a nilpotent Jordan matrix, either of which may have size zero.

\subsubsection{A map from $\Schwartz\left(\squareMatrix_n\left(\finiteField\right)\right)$ to $\Schwartz\left(\GL_n\left(\finiteField\right)\right)$}
We claim that applying $1-\eta_{\singular}$ to $\algebraIsomorphism\left(f\right)$ for $f \in \Schwartz\left(\squareMatrix_n\left(\finiteField\right)\right)$ results in a function whose Fourier transform vanishes on singular elements. Moreover, this resulting function agrees with $f$ on $\GL_n\left(\finiteField\right)$.
\begin{theorem}\label{thm:fourier-transform-of-kovacs-unit-projection}
	Let $f \in \Schwartz\left(\squareMatrix_n\left(\finiteField\right)\right)$, and set $f_0 = \algebraIsomorphism^{-1}\left(\eta_{\singular} \algebraIsomorphism\left(f\right)\right)$. Then
	\begin{enumerate}
		 \item \label{item:fourier-transform-agrees-on-singular-elements} For every singular $X \in \squareMatrix_n\left(\finiteField\right)$,
		 \begin{displaymath}
		 \fourierTransform{\fieldCharacter} f_0\left(X\right) = \fourierTransform{\fieldCharacter}f\left(X\right).
		 \end{displaymath}
		 \item \label{item:function-is-zero-on-regular-elements} For every $g \in \GL_n\left(\finiteField\right)$, $f_0\left(g\right) = 0$.
	\end{enumerate}
\end{theorem}
\begin{proof}
	\begin{enumerate}
		\item Let $X \in \squareMatrix_n\left(\finiteField\right)$ be singular. We have that
		\begin{displaymath}
		\fourierTransform{\fieldCharacter} f_0\left(X\right) = \evaluation_{\fieldCharacter}\left(\left[X\right] \algebraIsomorphism\left(f_0\right)\right) = \evaluation_{\fieldCharacter}\left(\left[X\right] \eta_{\singular} \algebraIsomorphism\left(f\right)\right).
		\end{displaymath}
		Since $X$ is singular, we have that $\left[X\right] \eta_{\singular} = \left[X\right]$ and thus
		\begin{displaymath}
		\fourierTransform{\fieldCharacter} f_0\left(X\right) = \evaluation_{\fieldCharacter}\left(\left[X\right] \algebraIsomorphism\left(f\right)\right) = \fourierTransform{\fieldCharacter} f\left(X\right).
		\end{displaymath}
		\item Since $\coefficientsField\left[\squareMatrix_n\left(\finiteField\right)\right]_{\singular} \subset \coefficientsField\left[\squareMatrix_n\left(\finiteField\right)\right]$ is an ideal, $\eta_{\singular} \algebraIsomorphism\left(f\right)$ lies in $\coefficientsField\left[\squareMatrix_n\left(\finiteField\right)\right]_{\singular}$ and thus $f_0$ is supported on singular elements.
	\end{enumerate}
\end{proof}

Let us use \Cref{thm:fourier-transform-of-kovacs-unit-projection} to define a map $\Schwartz\left(\squareMatrix_n\left(\finiteField\right)\right) \to \Schwartz\left(\GL_n\left(\finiteField\right)\right)$ as follows. If $f \in \Schwartz\left(\squareMatrix_n\left(\finiteField\right)\right)$ let us define $f^{\regular} = \fourierTransform{\fieldCharacter^{-1}} \left(\algebraIsomorphism^{-1}\left(\left(1 - \eta_{\singular}\right) \algebraIsomorphism\left(\fourierTransform{\fieldCharacter} f\right)\right)\right)$.
	\begin{proposition}\label{prop:properties-of-regular-function}
	The following properties hold:
	\begin{enumerate}
		\item $\left(f^{\regular}\right)^{\regular} = f^{\regular}$.
		\item \label{item:f-reg-is-a-schwartz-function-on-gln} $f^{\regular} \in \Schwartz\left(\GL_n\left(\finiteField\right)\right)$.
		\item \label{item:f-reg-fourier-transform-agrees-with-fourier-transform-of-f-on-gln} $\fourierTransform{\fieldCharacter} f^{\regular} \restriction_{\GL_n\left(\finiteField\right)} = \fourierTransform{\fieldCharacter} f \restriction_{\GL_n\left(\finiteField\right)}$.
		\item If the characteristic of $\coefficientsField$ does not divide $\sizeof{\GL_n\left(\finiteField\right)}$ then $f^{\regular}$ is the unique function satisfying (\ref{item:f-reg-is-a-schwartz-function-on-gln}) and (\ref{item:f-reg-fourier-transform-agrees-with-fourier-transform-of-f-on-gln}).
	\end{enumerate}
\end{proposition}
\begin{proof}
	\begin{enumerate}
		\item This follows from the Fourier inversion formula and from $\left(1-\eta_{\singular}\right)^2 = 1-\eta_{\singular}$, which in turn follows from $\eta_{\singular}^2 = \eta_{\singular}$.
		\item This follows from \Cref{thm:fourier-transform-of-kovacs-unit-projection} part (\ref{item:fourier-transform-agrees-on-singular-elements}) (invoked with $\fieldCharacter^{-1}$ instead of $\fieldCharacter$).
		\item This follows from \Cref{thm:fourier-transform-of-kovacs-unit-projection} part (\ref{item:function-is-zero-on-regular-elements}) and the Fourier inversion formula.
		\item This follows from \Cref{cor:fourier-transform-is-zero-on-gln-implies-f-is-zero}.
	\end{enumerate}
\end{proof}

As a corollary of \Cref{prop:properties-of-regular-function} and \Cref{thm:functional-equation-first-version}, we obtain the following regularized Godement--Jacquet functional equation.
\begin{corollary}
	For every irreducible representation $\pi$ of $\GL_n\left(\finiteField\right)$ and every $f \in \Schwartz\left(\squareMatrix_n\left(\finiteField\right)\right)$ we have
	\begin{displaymath}
	\sum_{g \in \GL_n\left(\finiteField\right)} \fourierTransform{\fieldCharacter} f\left(g^{-1}\right) \pi\left(g\right) = \GaussSumScalar{\pi}{\fieldCharacter} \sum_{g \in \GL_n\left(\finiteField\right)} f^{\regular}\left(g\right) \pi\left(g\right).
	\end{displaymath}
\end{corollary}

\subsubsection{Computation for indicator functions}
Let $A \in \squareMatrix_n\left(\finiteField\right)$ and take $f \in \Schwartz\left(\squareMatrix_n\left(\finiteField\right)\right)$ to be the indicator function of $A$. We have that $\fourierTransform{\fieldCharacter} f \left(X\right) = q^{-\frac{n^2}{2}}\fieldCharacter\left(\trace\left(AX\right)\right)$ for $X \in \squareMatrix_n\left(\finiteField\right)$. Notice that if $A \in \GL_n\left(\finiteField\right)$ then $f \in \Schwartz\left(\GL_n\left(\finiteField\right)\right)$ and from the first version of the functional equation we obtain the equality \begin{equation*}
	q^{-\frac{n^2}{2}} \sum_{g \in \GL_n\left(\finiteField\right)} \fieldCharacter\left(\trace\left(A g^{-1}\right)\right) \pi\left(g\right) = \GaussSumScalar{\pi}{\fieldCharacter} \pi\left(A\right),
\end{equation*}
which can also be easily verified directly by changing variables $g \mapsto g A$.

Let us compute $f^{\regular}$. We have that
\begin{displaymath}
\eta_{\singular} \algebraIsomorphism\left(\fourierTransform{\fieldCharacter} f\right) = -q^{-\frac{n^2}{2}} \mu\left(n\right)^{-1} \sum_{\substack{m \in \squareMatrix_n\left(\finiteField\right) \\
		m \text{ singular}\\
		m \text{ semi-idempotent}}} \sum_{X \in \squareMatrix_n\left(\finiteField\right)} \mu\left(\srank m\right) \left[m X\right] \fieldCharacter\left(\trace\left(A X\right)\right).
		\end{displaymath}
For $g \in \GL_n\left(\finiteField\right)$ we have that $f^{\regular}\left(g\right) = f\left(g\right) - \evaluation_{\fieldCharacter^{-1}}\left(\left[g\right] \eta_{\singular} \algebraIsomorphism\left(\fourierTransform{\fieldCharacter} f\right)\right)$,
which equals
\begin{displaymath}
f^{\regular}\left(g\right) = f\left(g\right) + q^{-n^2} \mu\left(n\right)^{-1} \sum_{\substack{m \in \squareMatrix_n\left(\finiteField\right) \\
		m \text{ singular}\\
		m \text{ semi-idempotent}}} \mu\left(\srank m\right) \sum_{X \in \squareMatrix_n\left(\finiteField\right)} \fieldCharacter\left(\trace\left(\left(A - g m\right)X\right)\right).
		\end{displaymath}
	The inner sum vanishes unless $gm = A$, in which case it equals $q^{n^2}$. Notice that if $A \in \GL_n\left(\finiteField\right)$, then $gm = A$ is impossible, because $m$ is singular. Thus if $A \in \GL_n\left(\finiteField\right)$ we have that $f^{\regular} = f$. If $A \notin \GL_n\left(\finiteField\right)$ we have $f\left(g\right) = 0$ and thus
	\begin{displaymath}
	f^{\regular}\left(g\right) = \begin{dcases}
	\mu\left(n\right)^{-1} \mu\left( \srank \left(g^{-1} A\right) \right) & \text{if } g^{-1}A \text{ is semi-idempotent,}\\
	0 & \text{otherwise.}
	\end{dcases}
	\end{displaymath}
	Notice that this formula is also true if $A \in \GL_n\left(\finiteField\right)$ since in that case $f^{\regular} = f$ and $g^{-1} A$ is semi-idempotent if and only if $A = g$.
	
	\subsubsection{Explicit regularized functional equation}\label{subsec:regularized-godement-jacquet-functional-equation}
	The computation above allows us to give a somewhat explicit formula for $f^{\regular}$ for a general function $f$.
	
	\begin{theorem}\label{thm:godement-jacquet-for-arbitrary-functions}
		For every $f \in \Schwartz\left(\squareMatrix_n\left(\finiteField\right)\right)$ and every $X \in \squareMatrix_n\left(\finiteField\right)$, we have
		
		\begin{displaymath}
		f^{\regular}\left(X\right) = \mu\left(n\right)^{-1} \sum_{\substack{m \in \squareMatrix_n\left(\finiteField\right)\\
	m \text{ semi-idempotent}}} \mu\left(\srank m\right) f\left(Xm\right).
	\end{displaymath}
	Thus for every irreducible representation $\pi$ of $\GL_n\left(\finiteField\right)$ and every $f \in \Schwartz\left(\squareMatrix_n\left(\finiteField\right)\right)$, we have the equality
	\begin{equation*}
		\begin{split}
			\sum_{g \in \GL_n\left(\finiteField\right)} \fourierTransform{\fieldCharacter}f\left(g^{-1}\right) \pi\left(g\right) =& \mu\left(n\right)^{-1} \GaussSumScalar{\pi}{\fieldCharacter} \sum_{g \in \GL_n\left(\finiteField\right)} \sum_{\substack{m \in \squareMatrix_n\left(\finiteField\right)\\
					m \text{ semi-idempotent}}} \mu\left(\srank m\right) f\left(gm\right) \pi\left(g\right).
		\end{split}
	\end{equation*}
	\end{theorem}
	\begin{proof}
		For $X \in \GL_n\left(\finiteField\right)$ this follows from the work above by writing $f$ as a linear combination of indicator functions. For $X \notin \GL_n\left(\finiteField\right)$ this follows from the fact that $\left[X\right]\left(1-\eta_{\singular}\right) = 0$ in $\coefficientsField\left[\squareMatrix_n\left(\finiteField\right)\right]$, and thus
		\begin{displaymath}
			\sum_{\substack{m \in \squareMatrix_n\left(\finiteField\right)\\
					m \text{ semi-idempotent}}} \mu\left(\srank m\right) \left[Xm\right] = 0.
		\end{displaymath}
		We thus have that $\rho'\left(\left[X\right]\left(1 - \eta_{\singular}\right), \left[\IdentityMatrix{n}\right]\right) f = 0$ which implies the result by evaluating the resulting function at $\IdentityMatrix{n}$.
	\end{proof}
	
	For what follows, it will be useful to introduce some notation. If $\pi$ is an irreducible representation of $\GL_n\left(\finiteField\right)$ and $A \in \squareMatrix_n\left(\finiteField\right)$, define $\pi^{\natural}\left(A\right) \in \EndomorphismRing_{\coefficientsField}\left(\pi\right)$ by the formula
	\begin{displaymath}
		\pi^{\natural}\left(A\right) = \sum_{g \in \GL_n\left(\finiteField\right)} \left(\delta_A\right)^{\regular}\left(g\right) \pi\left(g\right) =  \mu\left(n\right)^{-1}\sum_{g \in \GL_n\left(\finiteField\right)} \sum_{\substack{m \in \squareMatrix_n\left(\finiteField\right)\\
				m \text{ semi-idempotent}\\
				gm = A}} \mu\left(\srank m\right) \pi\left(g\right),
	\end{displaymath}
	where $\delta_A$ is the indicator function of $A$. By the results above we have that
	\begin{equation}\label{eq:expression-of-pi-natural-with-gauss-sums}
		\pi^{\natural}\left(A\right) = \GaussSumScalar{\pi}{\fieldCharacter}^{-1} q^{-\frac{n^2}{2}} \sum_{g \in \GL_n\left(\finiteField\right)} \fieldCharacter\left(\trace\left(A g^{-1}\right)\right) \pi\left(g\right).
	\end{equation}
	Also, as explained above, if $A \in \GL_n\left(\finiteField\right)$ then $\pi^{\natural}\left(A\right) = \pi\left(A\right)$. Thus $\pi^{\natural}$ can be thought of as an extension of $\pi$ to $\squareMatrix_n\left(\finiteField\right)$. However, generally speaking, $\pi^{\natural} \colon \squareMatrix_n\left(\finiteField\right) \to \EndomorphismRing_{\coefficientsField}\left(\pi\right)$ is not a homomorphism of monoids.
	It is, however, not difficult to show (either from the definition or from \eqref{eq:expression-of-pi-natural-with-gauss-sums}) that for $g_1, g_2 \in \GL_n\left(\finiteField\right)$ and $A \in \squareMatrix_n\left(\finiteField\right)$ we have
	\begin{equation}\label{eq:equivariance-of-pi-natural}
		\pi^{\natural}\left(g_1 A g_2\right) = \pi\left(g_1\right) \circ \pi^{\natural}\left(A\right) \circ \pi\left(g_2\right).
	\end{equation}
	
	Using this notation, we may restate \Cref{thm:godement-jacquet-for-arbitrary-functions} in the following form.
	\begin{theorem}\label{thm:gj-in-terms-of-pi-natural}
		For every $f \in \Schwartz\left(\squareMatrix_n\left(\finiteField\right)\right)$ and every irreducible representation $\pi$ of $\GL_n\left(\finiteField\right)$ we have
		\begin{displaymath}
			\sum_{g \in \GL_n\left(\finiteField\right)} \fourierTransform{\fieldCharacter}f\left(g^{-1}\right) \pi\left(g\right) = \GaussSumScalar{\pi}{\fieldCharacter} \sum_{X \in \squareMatrix_n\left(\finiteField\right)} f\left(X\right) \pi^{\natural}\left(X\right).
		\end{displaymath}
	\end{theorem}
	
	\begin{remark}
		The non-vanishing of the operator $\pi^{\natural}\left(A\right)$ for a singular matrix $A$ is closely related to the work of Gurevich--Howe on the theta correspondence for finite general linear groups~\cite{gurevich2021harmonic}. If $\rank A = r$ then $\pi^{\natural}\left(A\right)$ is zero unless $\pi$ has strict tensor co-rank $\ge n-r$. See \Cref{sec:representations-with-strict-tensor-corank}.
	\end{remark}
	\begin{remark}
		If the characteristic of $\coefficientsField$ divides $q-1$, then by~\cite[Theorem 4.8]{KurinczukMatringeSecherre2026} the Godement--Jacquet functional equation holds for any irreducible representation $\pi$ and any function $f \in \Schwartz\left(\squareMatrix_n\left(\finiteField\right)\right)$, without any regularization. By substituting the indicator function $\delta_A$ in \Cref{thm:gj-in-terms-of-pi-natural}, it follows in this case that $\pi^{\natural}\left(A\right) = 0$ for any singular $A$.
	\end{remark}
	
	\section{Results for $\coefficientsField\left[\squareMatrix_n\left(\finiteField\right)\right]$}
	
	In this section, we use the functional equation we established above to obtain results for the representation theory of the monoid algebra $\coefficientsField\left[\squareMatrix_n\left(\finiteField\right)\right]$. These results include a Godement--Jacquet functional equation and an explicit formula for primitive central idempotents corresponding to irreducible representations of $\coefficientsField\left[\squareMatrix_n\left(\finiteField\right)\right]$.
	
	\subsection{An overview of the representation theory of $\coefficientsField\left[\squareMatrix_n\left(\finiteField\right)\right]$}\label{subsec:overview-of-rep-theory-of-mn-fq}
	We quickly recall fundamental results regarding the representation theory of $\coefficientsField\left[\squareMatrix_n\left(\finiteField\right)\right]$.
	
	\subsubsection{Structure of $\coefficientsField\left[\squareMatrix_n\left(\finiteField\right)\right]$}
	We follow~\cite[Section 2.2]{HarmanSnowdenZelingher2025} to describe a decomposition of the monoid algebra $\coefficientsField\left[\squareMatrix_n\left(\finiteField\right)\right]$. Let $0 \le r \le n$. Set
	\begin{displaymath}
	\coefficientsField\left[\mathcal{G}_{n,r}\right] = \bigoplus_{\substack{V,W \subset \finiteField^n\\
	\dim V = \dim W = r}} \coefficientsField\left[\Isom\left(V, W\right)\right],
	\end{displaymath}
	where the direct sum is over all linear subspaces $V$ and $W$ of dimension $r$. Here $\Isom\left(V, W\right)$ is the set consisting of $\finiteField$-linear isomorphisms from $V$ to $W$, and $\coefficientsField\left[\Isom\left(V,W\right)\right]$ is the $\coefficientsField$-linear span of $\Isom\left(V, W\right)$. If $T \colon V \to W$ is an isomorphism, we write $\left[T\right]$ for the corresponding element in $\coefficientsField\left[\Isom\left(V,W\right)\right]$. We define multiplication on $\coefficientsField\left[\mathcal{G}_{n,r}\right]$ as follows. Let $T_1 \colon V_1 \to W_1$ and $T_2 \colon V_2 \to W_2$ be isomorphisms. Then we set $\left[T_1\right] \cdot \left[T_2\right] = \left[T_1 \circ T_2\right]$ if $W_2 = V_1$ and $\left[T_1\right] \cdot \left[T_2\right] = 0$ otherwise. We then extend multiplication bilinearly.
	
	Let $\coefficientsField\left[\mathcal{G}_n\right] = \bigoplus_{r = 0}^{n} \coefficientsField\left[\mathcal{G}_{n,r}\right]$. We equip $\coefficientsField\left[\mathcal{G}_n\right]$ with multiplication by setting the different summands in the direct sum to be orthogonal.
	
	For every $0 \le r \le n$ let $\Psi_r \colon \coefficientsField\left[\squareMatrix_n\left(\finiteField\right)\right] \to \coefficientsField\left[\mathcal{G}_{n,r}\right]$ be the $\coefficientsField$-linear map defined on basis elements by
	\begin{displaymath}
	\Psi_r\left(\left[m\right]\right) = \sum_{\substack{V \subset \finiteField^n\\
	\dim V = r\\
	\dim m\left(V\right) = r}} \left[m \restriction_{V}\right].
	\end{displaymath}
	Note that the sum is over all $V$ such that $m \restriction_V \colon V \to m\left(V\right)$ is an isomorphism.
	
	By~\cite[Proposition 2.2]{HarmanSnowdenZelingher2025} (see also~\cite{Kovacs1992}) we have that $\Psi_r$ is a homomorphism of rings and that the $\coefficientsField$-linear map $\Psi \colon \coefficientsField\left[\squareMatrix_n\left(\finiteField\right)\right] \to \coefficientsField\left[\mathcal{G}_{n}\right]$ given by the formula $\Psi = \bigoplus_{r=0}^n \Psi_r$ is an isomorphism of rings. In~\cite[Section 4]{HarmanSnowdenZelingher2025} we also described the inverse map, which we explain in the following theorem.
	
	\begin{theorem}\label{thm:inverse-image-of-Psi}
		Let $V \subset \finiteField^n$ be a linear subspace of dimension $r$, where $0 \le r \le n$. Define
		\begin{displaymath}
		\eta_V = q^{-r \left(n - r\right)} \mu\left(r\right)^{-1} \sum_{\substack{S \colon \finiteField^n \to \finiteField^n\\
				S \text{ semi-idempotent}\\
				\ImageOf S \subset V}} \mu\left(\srank S\right) \left[S\right].
				\end{displaymath}
		Let $T \colon V \to W$ be an isomorphism and let $T' \colon \finiteField^n \to \finiteField^n$ be any linear map such that $T' \restriction_{V} = T$. Then $\Psi\left(\left[T'\right] \eta_V\right) = \left[T\right]$.
	\end{theorem}
	\begin{proof}
		By~\cite[Theorem 4.2]{HarmanSnowdenZelingher2025} we have that $\Psi_j\left(\eta_V\right) = 0$ for $j > \dim V$ and that $\Psi_j\left(\eta_V\right) = \left[\idmap_V\right]$ for $j = \dim V$. It follows that for $j = \dim V$
		\begin{displaymath}
		\Psi_j\left(\left[T'\right]\eta_V\right)= \Psi_j\left(\left[T'\right]\right) \Psi_j\left(\eta_V\right) = \sum_{\substack{U \subset \finiteField^n \\
		\dim U = j\\
		\dim T'\left(U\right) = j}} \left[T' \restriction_U \right] \cdot \left[\idmap_{V}\right].
		\end{displaymath}
	In order for the summand to be non-zero we must have $U = V$ and thus the sum equals $\left[T' \restriction_V \circ \idmap_V\right] = \left[T\right]$.
	
	It remains to show that $\Psi_j\left(\left[T'\right]\eta_V\right) = 0$ for $j < \dim V$. It suffices to show that $\Psi_j\left(\eta_V\right) = 0$ for $j < \dim V$. Let
	\begin{displaymath}
	\eta_j = \sum_{\substack{W \subset \finiteField^n\\
			\dim W \le j}} \eta_W.
			\end{displaymath}
	Also define $\eta_{-1} = 0$.
	Let us write
	\begin{displaymath}
	\Psi_j\left(\eta_V\right) = \sum_i c_i \left[T_i\right],
	\end{displaymath}
	where $T_i \colon U_i \to U'_i$ are isomorphisms of $j$-dimensional linear subspaces of $\finiteField^n$ and $c_i \in \coefficientsField$. Then
	\begin{equation*}
		\begin{split}
			\Psi_j\left(\eta_V \left(\eta_j - \eta_{j-1}\right)\right) &= \Psi_j\left(\eta_V\right)\Psi_j\left(\eta_j-\eta_{j-1}\right) \\
			&= \sum_i \sum_{\substack{W \subset \finiteField^n\\
					\dim W = j}} c_i \left[T_i\right] \cdot \left[\idmap_W\right]  = \sum_i c_i \left[T_i\right] = \Psi_j\left(\eta_V\right),
		\end{split}
	\end{equation*}
	where we used the fact that $\Psi_j$ is a homomorphism. By~\cite[Proposition 4.9]{HarmanSnowdenZelingher2025} we have that $\left(\eta_W\right)_W$ are pairwise orthogonal (where $W$ ranges over all linear subspaces of $\finiteField^n$). Since $\dim V > j$, $\eta_V$ is orthogonal to every $\eta_W$ in the definitions of $\eta_j$ and $\eta_{j-1}$ and thus $\eta_V \eta_j = \eta_V \eta_{j-1} = 0$ which implies $\Psi_j\left(\eta_V\right) = 0$.
	\end{proof}

	\subsubsection{Representations of $\coefficientsField\left[\squareMatrix_n\left(\finiteField\right)\right]$}\label{subsec:representations-of-m-n-f-q}
	The decomposition above allows us to identify $\coefficientsField\left[\squareMatrix_n\left(\finiteField\right)\right]$ with the direct sum $\bigoplus_{r=0}^{n} \squareMatrix_{N\left(r\right)}\left(\coefficientsField\left[\GL_r\left(\finiteField\right)\right]\right)$, where $N\left(r\right) = \binom{n}{r}_q$ is the number of $r$-dimensional linear subspaces of $\finiteField^n$. This is done as follows.
	
	For every linear subspace $V \subset \finiteField^n$ of dimension $r$ fix an isomorphism $\phi_{V} \colon V \to \finiteField^r$. Given $r$-dimensional linear subspaces $V$ and $W$ of $\finiteField^n$ let $E_{W,V} \in \squareMatrix_{N\left(r\right)}\left(\coefficientsField\left[\GL_r\left(\finiteField\right)\right]\right)$ be a matrix whose rows and columns are indexed by $r$-dimensional linear subspaces of $\finiteField^n$, consisting of zero everywhere, except for the entry indexed by $\left(W,V\right)$ where the value is $1$. Given an $\finiteField$-linear isomorphism $T \colon V \to W$ of $r$-dimensional linear subspaces of $\finiteField^n$, we map $\left[T\right]$ to the matrix $\left[\phi_{W} \circ T \circ \phi_{V}^{-1}\right] E_{W,V}$. We extend this map to a map $\coefficientsField\left[\mathcal{G}_{n,r}\right] \to \squareMatrix_{N\left(r\right)}\left(\coefficientsField\left[\GL_r\left(\finiteField\right)\right]\right)$ by linearity. This induces an isomorphism of rings $\coefficientsField\left[\mathcal{G}_{n}\right] \to \bigoplus_{r=0}^{n} \squareMatrix_{N\left(r\right)}\left(\coefficientsField\left[\GL_r\left(\finiteField\right)\right]\right)$. 
	
	By~\cite[Example A.27]{Steinberg2016}, the category of $\squareMatrix_{N\left(r\right)}\left(\coefficientsField\left[\GL_r\left(\finiteField\right)\right]\right)$-modules is equivalent to the category of $\coefficientsField\left[\GL_r\left(\finiteField\right)\right]$-modules. Given a representation $\left(\sigma, \mathrm{V}_{\sigma}\right)$ of $\GL_r\left(\finiteField\right)$, the corresponding $\squareMatrix_{N\left(r\right)}\left(\coefficientsField\left[\GL_r\left(\finiteField\right)\right]\right)$-module is given by $L_n\left(\sigma\right) = \mathrm{V}_{\sigma} \otimes \coefficientsField^{N\left(r\right)}$ with action defined on pure tensors by
	\begin{displaymath}
	A \cdot \left(v_{\sigma} \otimes b_W\right) = \sum_{V} \left(\sigma\left(a_{VW}\right) v_{\sigma}\right) \otimes b_V,
	\end{displaymath}
	where $v_{\sigma} \in \mathrm{V}_{\sigma}$, $A = \left(a_{V,W}\right)_{V,W} \in \squareMatrix_{N\left(r\right)}\left(\coefficientsField\left[\GL_r\left(\finiteField\right)\right]\right)$, and where $\left(b_W\right)_{W}$ is the standard basis of $\coefficientsField^{N\left(r\right)}$. Here $V$ and $W$ range over all $r$-dimensional linear subspaces of $\finiteField^n$ and $\sigma$ is realized as a representation of $\coefficientsField\left[\GL_r\left(\finiteField\right)\right]$ in the usual way, by mapping $\left[g\right] \mapsto \sigma\left(g\right)$ for $g \in \GL_r\left(\finiteField\right)$ and extending linearly to arbitrary elements.
	
	\subsubsection{Dual space of a representation}
	Suppose that $\left(\sigma, \mathrm{V}_{\sigma}\right)$ is an irreducible representation of $\GL_r\left(\finiteField\right)$ for $0 \le r \le n$ and let $\Pi = L_n\left(\sigma\right)$.
	
	As explained above, $\Pi$ can be realized as the tensor product $\mathrm{V}_{\sigma} \otimes_{\coefficientsField} \coefficientsField^{N\left(r\right)}$. We let $\Pi^{\vee} = \Hom_{\coefficientsField}\left(\Pi, \coefficientsField\right)$. Then $\Pi^{\vee}$ can be identified with $L_n\left(\sigma^{\vee}\right)$ where $\sigma^{\vee}$ is the dual representation of $\sigma$ as follows. For $v_{\sigma} \in \sigma$, $v_{\sigma^{\vee}} \in \sigma^{\vee}$ and $r$-dimensional linear subspaces $V, V'$ of $\finiteField^n$ we set
	\begin{displaymath}
		\innerproduct{v_{\sigma} \otimes b_{V}}{v_{\sigma^{\vee}} \otimes b_{V'}} = \innerproduct{v_{\sigma}}{v_{\sigma^{\vee}}} \delta_{V, V'}.
	\end{displaymath}
	Here $\innerproduct{v_{\sigma}}{v_{\sigma^{\vee}}}$ is the standard evaluation map and $\delta_{V,V'}$ is the Kronecker delta function.
	
	Using this identification, we have two actions on $\Pi^{\vee}$. The first action is $L_n\left(\sigma^{\vee}\right)$. The second action is $\transpose{\Pi}$ given by $\innerproduct{v_{\Pi}}{\transpose{\Pi}\left(\left[X\right]\right) v_{\Pi^{\vee}}} = \innerproduct{\Pi\left(\left[\transpose{X}\right]\right) v_{\Pi}}{v_{\Pi^{\vee}}}$ for every $X \in \squareMatrix_n\left(\finiteField\right)$ and every $v_{\Pi} \in \Pi$ and $v_{\Pi^{\vee}} \in \Pi^{\vee}$. These actions are usually not identical. For $X \in \GL_n\left(\finiteField\right)$ they are related via the relation $\transpose{\Pi}\left(\left[X\right]\right) = L_n\left(\sigma^{\vee}\right)\left(\left[X^{\iota}\right]\right)$.
	
	\subsection{Godement--Jacquet functional equation for representations of $\coefficientsField\left[\squareMatrix_n\left(\finiteField\right)\right]$}
	In this section we provide a Godement--Jacquet functional equation for irreducible representations of $\coefficientsField\left[\squareMatrix_n\left(\finiteField\right)\right]$. We motivate this with a discussion regarding the equivariance property satisfied by the Godement--Jacquet zeta integrals.
	
	Let us first recall the situation in the group case. In this case, for any irreducible representation $\pi$ of $\GL_n\left(\finiteField\right)$, the Godement--Jacquet zeta integral produces a non-zero element of the following Hom-space
	\begin{displaymath}
		\Hom_{\GL_n\left(\finiteField\right) \times \GL_n\left(\finiteField\right)}\left(\left(\pi \otimes \pi^{\vee}\right) \otimes \Schwartz\left(\GL_n\left(\finiteField\right)\right), \coefficientsField\right).
	\end{displaymath}
	This element is the assignment
	$B_{\pi} \colon \pi \otimes \pi^{\vee} \otimes \Schwartz\left(\GL_n\left(\finiteField\right)\right) \to \coefficientsField$ defined by the formula
	\begin{displaymath}
		B_{\pi}\left(v_{\pi} \otimes v_{\pi^{\vee}}, f\right) = \sum_{g \in \GL_n\left(\finiteField\right)} f\left(g\right) \innerproduct{\pi\left(g\right) 	v_{\pi}}{v_{\pi^{\vee}}}.
	\end{displaymath}
	By elementary representation theory of finite groups, we know that the Hom-space in question is one-dimensional. The dual Godement--Jacquet zeta produces another non-zero element of this space $B^{\ast}_{\pi} \colon \pi \otimes \pi^{\vee} \otimes \Schwartz\left(\GL_n\left(\finiteField\right)\right) \to \coefficientsField$, given by
	\begin{displaymath}
		B^{\ast}_{\pi}\left(v_{\pi} \otimes 	v_{\pi^{\vee}}, f\right) = \sum_{g \in \GL_n\left(\finiteField\right)} \fourierTransform{\fieldCharacter}f\left(g\right) \innerproduct{v_{\pi}}{\pi^{\vee}\left(g\right) v_{\pi^{\vee}}}.
	\end{displaymath}
	Since $B_\pi$ and $B^{\ast}_{\pi}$ both lie in a one-dimensional space, they are proportional. This gives rise to the gamma factor $\gamma\left(\pi, \fieldCharacter\right)$.
	
	In the discussion above, in order to show that $B_{\pi}^{\ast}$ lies in the relevant Hom-space, one uses the property \begin{displaymath}
		\fourierTransform{\fieldCharacter} \rho\left(g, h\right) f = \rho\left(h, g\right) \fourierTransform{\fieldCharacter}f,
	\end{displaymath}
	for any $g,h \in \GL_n\left(\finiteField\right)$ and $f \in \Schwartz\left(\squareMatrix_n\left(\finiteField\right)\right)$. This property does not hold if one replaces $\rho$ with $\rho'$ and allows $g,h \in \squareMatrix_n\left(\finiteField\right)$.
	
	We move to discuss the analogous story for $\coefficientsField\left[\squareMatrix_n\left(\finiteField\right)\right]$. Let $\Pi = L_n\left(\sigma\right)$ where $\sigma$ is an irreducible representation of $\GL_r\left(\finiteField\right)$ for $0 \le r \le n$. Using $\zetaIntegral_{\Pi}$ we can define an element of the following one-dimensional Hom-space (see~\cite[Section 6]{HarmanSnowdenZelingher2025}):
	\begin{displaymath}
		\Hom_{\coefficientsField\left[\squareMatrix_n\left(\finiteField\right) \times \squareMatrix_n\left(\finiteField\right)\right]}\left(\Pi \otimes \transpose{\Pi}, \transpose{\left(\lambda', \coefficientsField\left[\squareMatrix_n\left(\finiteField\right)\right]\right)}\right)
	\end{displaymath}
	as follows. 
	
	Given $v_{\Pi} \in \Pi$, $v_{\Pi^{\vee}} \in \Pi^{\vee}$, and $f \in \Schwartz\left(\squareMatrix_n\left(\finiteField\right)\right)$, we define
	\begin{displaymath}
		B_{\Pi}\left(v_{\Pi}\otimes v_{\Pi^{\vee}}, \algebraIsomorphism\left(f\right)\right) = \innerproduct{\zetaIntegral_{\Pi}\left(f\right)v_{\Pi}}{v_{\Pi^{\vee}}} = \sum_{X \in \squareMatrix_n\left(\finiteField\right)} f\left(X\right) \innerproduct{\Pi\left(\left[X\right]\right) 	v_{\Pi}}{v_{\Pi^{\vee}}}.
	\end{displaymath}
	One can easily check that for any $A, B \in \squareMatrix_n\left(\finiteField\right)$
	\begin{displaymath}
		B_{\Pi}\left(\Pi\left(\left[A\right]\right)v_{\Pi}\otimes  \transpose{\Pi}\left(\left[B\right]\right) v_{\Pi^{\vee}}, \algebraIsomorphism\left(f\right)\right) = B_{\Pi}\left(v_{\Pi} \otimes v_{\Pi^{\vee}}, \lambda'\left(\left[\transpose{A}\right], \left[\transpose{B}\right]\right)\algebraIsomorphism\left(f\right)\right).
	\end{displaymath}
	Thus $B_{\Pi}$ defines an element in the relevant Hom-space by \begin{displaymath}
		v_{\Pi} \otimes v_{\Pi^{\vee}} \mapsto \left(\algebraIsomorphism\left(f\right) \mapsto \innerproduct{\zetaIntegral_{\Pi}\left(f\right)v_{\Pi}}{v_{\Pi^{\vee}}}\right).
	\end{displaymath}
	Note that in this case we cannot obtain from $\zetaIntegral_{\Pi}$ an element of \begin{displaymath}
		\Hom_{\coefficientsField\left[\squareMatrix_n\left(\finiteField\right) \times \squareMatrix_n\left(\finiteField\right)\right]}\left(\left(\Pi \otimes \Pi^{\vee}\right) \otimes \rho', \coefficientsField\right)
	\end{displaymath}
	nor of 
	\begin{displaymath}
		\Hom_{\coefficientsField\left[\squareMatrix_n\left(\finiteField\right) \times \squareMatrix_n\left(\finiteField\right)\right]}\left(\left(\Pi \otimes \Pi^{\vee}\right) \otimes \lambda', \coefficientsField\right).
	\end{displaymath}
	
	It is natural to ask what the companion of $B_{\Pi}$ should be. In \Cref{thm:godement-jacquet-for-monoid-algebras} we provide an answer to this question. This theorem tells us that a companion in this case is given by
	\begin{equation*}
		B^{\ast}_{\Pi}\left(v_{\Pi}\otimes v_{\Pi^{\vee}}, \algebraIsomorphism\left(f\right)\right) = \innerproduct{v_{\Pi}}{\naturalZetaIntegral_{\Pi^{\vee}}\left(\fourierTransform{\fieldCharacter}f\right)v_{\Pi^{\vee}}} = \sum_{X \in \squareMatrix_n\left(\finiteField\right)} \fourierTransform{\fieldCharacter}f\left(X\right) \innerproduct{v_{\Pi}}{\left(\Pi^{\vee}\right)^{\natural}\left(X\right)v_{\Pi^{\vee}}},
	\end{equation*}
	where $\left(\Pi^{\vee}\right)^{\natural}$ is defined using $\left(\sigma^{\vee}\right)^{\natural}$ from \Cref{subsec:regularized-godement-jacquet-functional-equation} in the next subsection.
	
	\subsubsection{The functional equation}\label{subsec:functional-equation-for-monoid-algebras}
	Let $\sigma$ be an irreducible representation of $\GL_r\left(\finiteField\right)$ for $0 \le r \le n$. Recall the definition of $\sigma^{\natural}$ from \Cref{subsec:regularized-godement-jacquet-functional-equation}. We use $\sigma^{\natural}$ to define a map $L_n\left(\sigma\right)^{\natural} \colon \squareMatrix_n\left(\finiteField\right) \to \EndomorphismRing_{\coefficientsField}\left(L_n\left(\sigma\right)\right)$ as follows. Recall that from \Cref{subsec:representations-of-m-n-f-q} we have that $L_n\left(\sigma\right) = \mathrm{V}_{\sigma} \otimes \coefficientsField^{N\left(r\right)}$. Let $A \in \squareMatrix_n\left(\finiteField\right)$. For $r$-dimensional linear subspaces $V$ and $W$ of $\finiteField^n$, let
	\begin{displaymath}
		\sigma^{\natural}_{W,V}\left(A\right) = \begin{dcases}
			\sigma^{\natural}\left(\phi_W \circ A\restriction_{V} \circ \phi_V^{-1}\right) & \text{if } \ImageOf A \subset W, \\
			0 & \text{otherwise.}
		\end{dcases}
	\end{displaymath}
	We define $L_n\left(\sigma\right)^{\natural}\left(A\right)$ by setting for every $v_{\sigma} \in \sigma$ and every linear $r$-dimensional subspace $V \subset \finiteField^n$ \begin{displaymath}
		L_n\left(\sigma\right)^{\natural}\left(A\right) \left(v_{\sigma} \otimes b_V\right) = \sum_{W} \left(\sigma^{\natural}_{W,V}\left(A\right) v_{\sigma}\right) \otimes b_W,
	\end{displaymath}
	where $W$ goes over all the $r$-dimensional linear subspaces of $\finiteField^n$.
	
	\begin{remark}
		Note that by definition $L_n\left(\sigma\right)$ is supported on matrices of rank $\ge r$. In contrast, by definition $L_n\left(\sigma\right)^{\natural}$ is supported on matrices of rank $\le r$.
	\end{remark}
	
	We have the following regularized Godement--Jacquet functional equation for the monoid algebra $\coefficientsField\left[\squareMatrix_n\left(\finiteField\right)\right]$.
	
	\begin{theorem}\label{thm:godement-jacquet-for-monoid-algebras}
		For every irreducible representation $\sigma$ of $\GL_r\left(\finiteField\right)$, where $0 \le r \le n$, and every function $f \in \Schwartz\left(\squareMatrix_n\left(\finiteField\right)\right)$ let
		\begin{displaymath}
			\zetaIntegral_{L_n\left(\sigma\right)}\left(f\right) = \sum_{X \in \squareMatrix_n\left(\finiteField\right)} f\left(X\right) L_n\left(\sigma\right)\left(\left[X\right]\right),
		\end{displaymath}
		and let
		\begin{displaymath}
			\naturalZetaIntegral_{L_n\left(\sigma\right)}\left(f\right) = \sum_{X \in \squareMatrix_n\left(\finiteField\right)} f\left(X\right) L_n\left(\sigma\right)^{\natural}\left(X\right).
		\end{displaymath}
		Then we have that $q^{\frac{\left(n-r\right)^2}{2}} \GaussSumScalar{\sigma}{\fieldCharacter} \naturalZetaIntegral_{L_n\left(\sigma\right)}\left(f\right)$ is the dual map of $\zetaIntegral_{L_n\left(\sigma^{\vee}\right)}\left(\fourierTransform{\fieldCharacter}f\right)$. Equivalently, for every pair of $r$-dimensional subspaces $V, V' \subset \finiteField^n$, every $v_{\sigma} \in \sigma$, and every $v_{\sigma^{\vee}} \in \sigma^{\vee}$, we have
		\begin{equation}\label{eq:explicit-duality-of-zeta-integrals}
			\innerproduct{v_{\sigma} \otimes b_V}{\zetaIntegral_{L_n\left(\sigma^{\vee}\right) }\left(\fourierTransform{\fieldCharacter}f\right)\left(v_{\sigma^{\vee}} \otimes b_{V'}\right)} = q^{\frac{\left(n-r\right)^2}{2}} \GaussSumScalar{\sigma}{\fieldCharacter} \innerproduct{\naturalZetaIntegral_{L_n\left(\sigma\right)}\left(f\right) \left(v_{\sigma} \otimes b_{V}\right)}{v_{\sigma^{\vee}} \otimes b_{V'}}.
		\end{equation}
	\end{theorem}
	\begin{remark}
		Substituting $f = \delta_A$ in \Cref{thm:godement-jacquet-for-monoid-algebras}, we see that $L_n\left(\sigma\right)^{\natural}$ is related to $L_n\left(\sigma^{\vee}\right)$ via Fourier transform. To be precise, $L_n\left(\sigma\right)^{\natural}\left(A\right)$ is the dual linear map corresponding to the linear map $L_n\left(\sigma^{\vee}\right) \to L_n\left(\sigma^{\vee}\right)$ given by
		\begin{displaymath}
			q^{-\frac{\left(n-r\right)^2}{2}} \GaussSumScalar{\sigma}{\fieldCharacter}^{-1} \cdot q^{-\frac{n^2}{2}} \sum_{X \in \squareMatrix_n\left(\finiteField\right)} L_n\left(\sigma^{\vee}\right)\left(\left[X\right]\right) \fieldCharacter\left(\trace\left(A X\right)\right).
		\end{displaymath}
	\end{remark}
	\begin{proof}[Proof of \Cref{thm:godement-jacquet-for-monoid-algebras}]
		It suffices to prove this for $f = \delta_{A}$ for a matrix $A \in \squareMatrix_n\left(\finiteField\right)$. In this case 
		\begin{displaymath}
			\zetaIntegral_{L_n\left(\sigma^{\vee}\right)}\left(\fourierTransform{\fieldCharacter}\delta_A\right) = q^{-\frac{n^2}{2}} \sum_{X \in \squareMatrix_n\left(\finiteField\right)} \fieldCharacter\left(\trace\left(AX\right)\right) L_n\left(\sigma^{\vee}\right)\left(\left[X\right]\right).
		\end{displaymath}
		Let $V'$ be a $r$-dimensional linear subspace of $\finiteField^n$ and let $v_{\sigma^{\vee}} \in \sigma^{\vee}$. We have
		\begin{equation*}
			\begin{split}
				&\sum_{X \in \squareMatrix_n\left(\finiteField\right)} \fieldCharacter\left(\trace\left(AX\right)\right) L_n\left(\sigma^{\vee}\right)\left(\left[X\right]\right)\left(v_{\sigma^{\vee}} \otimes b_{V'}\right) \\
				&= \sum_{W}  \left(\sum_{\substack{X \in \squareMatrix_n\left(\finiteField\right)\\
						X V' = W}} \fieldCharacter\left(\trace\left(AX\right)\right) \sigma^{\vee}\left(\phi_W \circ X \restriction_{V'} \circ \phi_{V'}^{-1}\right)v_{\sigma^{\vee}} \right) \otimes b_W,
			\end{split}
		\end{equation*}
		where $W$ goes over all $r$-dimensional linear subspaces of $\finiteField^n$.
		
		Let us consider the inner sum
		\begin{displaymath}
			\sum_{\substack{X \in \squareMatrix_n\left(\finiteField\right)\\
					X V' = W}} \fieldCharacter\left(\trace\left(AX\right)\right) \sigma^{\vee}\left(\phi_W \circ X \restriction_{V'} \circ \phi_{V'}^{-1}\right)v_{\sigma^{\vee}}.
		\end{displaymath}
		We may rewrite this as
		\begin{equation}\label{eq:sum-over-B}
			\sum_{B \in \GL_r\left(\finiteField\right)}
			\sum_{\substack{X \in \squareMatrix_n\left(\finiteField\right) \\
					X \restriction_{V'} = \phi_W^{-1} \circ B \circ \phi_{V'}}} \fieldCharacter\left(\trace\left(AX\right)\right) \sigma^{\vee}\left(B\right)v_{\sigma^{\vee}}.
		\end{equation}
		Fix a decomposition $\finiteField^n = V' \oplus U$ and a linear isomorphism $T_{V',W} \colon \finiteField^n \to \finiteField^n$ such that $T_{V',W} \restriction_{V'} = \phi_W^{-1} \circ \phi_{V'}$. For linear maps $S_{V'} \colon V' \to \finiteField^n$ and $S_U \colon U \to \finiteField^n$ let $j\left(S_{V'},S_U\right)$ be the linear map whose restriction to $V'$ is $S_{V'}$ and whose restriction to $U$ is $S_U$. Then we may parameterize $X$ in the sum by $X = T_{V',W} \circ j\left(\phi_{V'}^{-1} \circ B \circ \phi_{V'}, S_U\right)$, where $S_U \colon U \to \finiteField^n$ is a linear map. Writing $S_U = S_{U,V'} \oplus S_{U,U}$, where $S_{U,V'} \colon U \to V'$ and $S_{U,U} \colon U \to U$ we see that the sum over $S_{U,U}$ vanishes unless $\ImageOf\left(A \circ T_{V',W} \restriction_{U} \right) \subset V'$ and that the sum over $S_{U,V'}$ vanishes unless $\ImageOf\left(A \circ T_{V',W} \restriction_{V'} \right) \subset V'$. These two conditions combined are equivalent to $\ImageOf\left(A \circ T_{V',W}\right) \subset V'$ which in turn is equivalent to $\ImageOf\left(A\right) \subset V'$. Thus if $\ImageOf A \not\subset V'$ we have $\zetaIntegral_{L_n\left(\sigma^{\vee}\right)}\left(\fourierTransform{\fieldCharacter}\delta_A\right)\left(v_{\sigma^{\vee}} \otimes b_{V'}\right) = 0$. Notice that by definition we also have in this case that $\innerproduct{\naturalZetaIntegral_{L_n\left(\sigma\right)}\left(\delta_A\right)\left(v_{\sigma} \otimes b_V\right)}{v_{\sigma^{\vee}} \otimes b_{V'}}=0$ and thus the equality \eqref{eq:explicit-duality-of-zeta-integrals} holds.
		
		Assume from now on that $\ImageOf A \subset V'$. From the discussion above we have that \eqref{eq:sum-over-B} equals
		\begin{displaymath}
			q^{n\left(n-r\right)} \sum_{B \in \GL_r\left(\finiteField\right)}
			\fieldCharacter\left(\trace\left(A \circ j\left(\phi_W^{-1} \circ B \circ \phi_{V'}, 0\right)\right)\right)  \sigma^{\vee}\left(B\right)v_{\sigma^{\vee}}.
		\end{displaymath}
		Since $\ImageOf A \subset V'$, we may rewrite the last formula as
		\begin{equation*}
			q^{n\left(n-r\right)} \sum_{B \in \GL_r\left(\finiteField\right)}
			\fieldCharacter\left(\trace\left(\phi_{V'} \circ A \restriction_W \circ \phi_W^{-1} \circ B\right)\right)  \sigma^{\vee}\left(B\right)v_{\sigma^{\vee}}.
		\end{equation*}
		Thus we have
		\begin{equation}\label{eq:sum-over-B-kondo-gauss-sum-form}
			\begin{split}
				&\zetaIntegral_{L_n\left(\sigma^{\vee}\right)}\left(\fourierTransform{\fieldCharacter}\delta_A\right)\left(v_{\sigma^{\vee}} \otimes b_{V'}\right)
				\\
				=& q^{-\frac{n^2}{2} + n\left(n-r\right)} \sum_{W} \sum_{B \in \GL_r\left(\finiteField\right)}
				\fieldCharacter\left(\trace\left(\phi_{V'} \circ A \restriction_W \circ \phi_W^{-1} \circ B\right)\right)  \sigma^{\vee}\left(B\right)v_{\sigma^{\vee}} \otimes b_W
			\end{split}
		\end{equation}
		It follows from \eqref{eq:sum-over-B-kondo-gauss-sum-form} that we have
		\begin{equation}\label{eq:pairing-with-GJ-of-fourier-transform}
			\begin{split}
				& \innerproduct{v_{\sigma} \otimes b_V}{\zetaIntegral_{L_n\left(\sigma^{\vee}\right)}\left(\fourierTransform{\fieldCharacter}\delta_A\right)\left(v_{\sigma^{\vee}} \otimes b_{V'}\right)} \\
				=& q^{-\frac{n^2}{2} + n\left(n-r\right)} \sum_{B \in \GL_r\left(\finiteField\right)}
				\innerproduct{\fieldCharacter\left(\trace\left(\phi_{V'} \circ A \restriction_V \circ \phi_V^{-1} \circ B\right)\right) \sigma\left(B^{-1}\right) v_{\sigma}}{v_{\sigma^{\vee}}}.
			\end{split}
		\end{equation}
		By \eqref{eq:expression-of-pi-natural-with-gauss-sums} evaluated with $\pi=\sigma$ and $n=r$, we have that \eqref{eq:pairing-with-GJ-of-fourier-transform} equals
		\begin{displaymath}
			q^{\frac{r^2-n^2}{2} + n\left(n-r\right)} \GaussSumScalar{\sigma}{\fieldCharacter} \innerproduct{L_n\left(\sigma\right)^{\natural}\left(A\right) \left(v_{\sigma} \otimes b_V\right)}{v_{\sigma^{\vee}} \otimes b_{V'}},
		\end{displaymath}
		which equals
		\begin{displaymath}
			q^{\frac{\left(n-r\right)^2}{2}} \GaussSumScalar{\sigma}{\fieldCharacter} \innerproduct{\naturalZetaIntegral_{L_n\left(\sigma\right)}\left(\delta_A\right) \left(v_{\sigma} \otimes b_V\right)}{v_{\sigma^{\vee}} \otimes b_{V'}},
		\end{displaymath}
		as required.
	\end{proof}	
			
	\subsection{Central idempotents}
	Throughout this section we assume that $\coefficientsField$ has characteristic not dividing $\sizeof{\GL_n\left(\finiteField\right)}$. Our goal is to write down an explicit formula for the primitive central idempotent in the monoid algebra $\coefficientsField\left[\squareMatrix_n\left(\finiteField\right)\right]$ corresponding to $L_n\left(\sigma\right)$ for an irreducible representation $\sigma$ of $\GL_r\left(\finiteField\right)$.
	
	\subsubsection{$\GL_n\left(\finiteField\right)$-invariant functions}
	Let $n \ge 0$. Recall that a \emph{class function} of $\GL_n\left(\finiteField\right)$ is a function $f \colon \GL_n\left(\finiteField\right) \to \coefficientsField$ that is invariant under conjugation.
	
	By a \emph{$\GL_n\left(\finiteField\right)$-invariant function} we mean a function $f \colon \squareMatrix_n\left(\finiteField\right) \to \coefficientsField$ such that $f\left(g^{-1} A g\right) = f\left(A\right)$ for every $g \in \GL_n\left(\finiteField\right)$ and $A \in \squareMatrix_n\left(\finiteField\right)$.
	
	\begin{remark}
		We do not use the term ``class function'' for $\squareMatrix_n\left(\finiteField\right)$ because it is defined differently in the literature, see~\cite[Chapter 7]{Steinberg2016}.
	\end{remark}
	
	If $V$ is an $n$-dimensional $\finiteField$-linear space and $f$ is a class function of $\GL_n\left(\finiteField\right)$ (respectively, a $\GL_n\left(\finiteField\right)$-invariant function) and $T \in \GL\left(V\right)$ (respectively, $T \in \EndomorphismRing\left(V\right)$) then we define $f\left(T\right)$ by choosing a basis for $V$ and evaluating $f$ at the matrix representing $T$ with respect to this basis. This definition does not depend on the chosen basis because $f$ is invariant under $\GL_n\left(\finiteField\right)$ conjugation.
	
	The class functions of $\GL_n\left(\finiteField\right)$ we will be mostly interested in are irreducible characters, i.e., functions of the form $g \mapsto \trace \pi\left(g\right)$ where $\pi$ is an irreducible representation of $\GL_n\left(\finiteField\right)$. The $\GL_n\left(\finiteField\right)$-invariant functions we will mostly be interested in are non-abelian Gauss sums, which we define next.
	
	\subsubsection{Degenerate non-abelian Gauss sums}\label{subsubsec:degenerate-non-abelian-gauss-sums}
	Given an irreducible representation $\pi$ of $\GL_n\left(\finiteField\right)$, and $A \in \squareMatrix_n\left(\finiteField\right)$, we define
	\begin{displaymath}
		\GaussSumScalarFunction{\pi}{\fieldCharacter}\left(A\right) = q^{-\frac{n^2}{2}} \sum_{g \in \GL_n\left(\finiteField\right)} \fieldCharacter\left(\trace\left(A g^{-1}\right)\right) \trace \pi\left(g\right).
	\end{displaymath}
	From \eqref{eq:expression-of-pi-natural-with-gauss-sums}, we have that
	\begin{equation}\label{eq:gauss-sum-function-in-terms-of-pi-natural}
		\GaussSumScalarFunction{\pi}{\fieldCharacter}\left(A\right) = \GaussSumScalar{\pi}{\fieldCharacter} \trace \pi^{\natural}\left(A\right).
	\end{equation}
	It follows from \eqref{eq:equivariance-of-pi-natural} and \eqref{eq:gauss-sum-function-in-terms-of-pi-natural} that $\GaussSumScalarFunction{\pi}{\fieldCharacter}$ is a $\GL_n\left(\finiteField\right)$-invariant function.
	
	\subsubsection{Central primitive idempotents}\label{subsubsec:central-primitive-idempotents}
	
	Notice that the center of $\bigoplus_{r=0}^n \squareMatrix_{N\left(r\right)}\left(\coefficientsField\left[\GL_r\left(\finiteField\right)\right]\right)$ is $\bigoplus_{r=0}^n Z\left(\coefficientsField\left[\GL_r\left(\finiteField\right)\right]\right) \IdentityMatrix{N\left(r\right)}$, where $Z\left(\coefficientsField\left[\GL_r\left(\finiteField\right)\right]\right)$ is the center of $\coefficientsField\left[\GL_r\left(\finiteField\right)\right]$.
	
	Let $\sigma$ be an irreducible representation of $\GL_r\left(\finiteField\right)$. Recall that the (central) idempotent in $\coefficientsField\left[\GL_r\left(\finiteField\right)\right]$ corresponding to $\sigma$ is given by the formula
	\begin{equation}\label{eq:idempotent-for-irreducible-representation-of-group}
		e_{\sigma} = \frac{\dim \sigma}{\sizeof{\GL_r\left(\finiteField\right)}} \sum_{g \in \GL_r\left(\finiteField\right)} \trace \sigma^{\vee}\left(g\right) \left[g\right].
	\end{equation}
	
	Given the description above of $\squareMatrix_{N\left(r\right)}\left(\coefficientsField\left[\GL_r\left(\finiteField\right)\right]\right)$-modules in terms of representations of $\GL_r\left(\finiteField\right)$, it is easy to see that the central element
	\begin{displaymath}
	f_{\sigma} = \diag\left(e_{\sigma},\dots,e_{\sigma}\right) \in \squareMatrix_{N\left(r\right)}\left(\coefficientsField\left[\GL_r\left(\finiteField\right)\right]\right)
	\end{displaymath}
	is idempotent. It is also clear that it acts by zero on $L_n\left(\sigma'\right)$ for every irreducible representation $\sigma'$ of $\GL_{r'}\left(\finiteField\right)$ that is not isomorphic to $\sigma$, where $0 \le r' \le n$ (if $r' \ne r$ we regard $\sigma$ and $\sigma'$ as non-isomorphic). It is also easy to see that $f_{\sigma}$ acts as the identity on $L_n\left(\sigma\right)$. Thus $f_{\sigma} \in \bigoplus_{j=0}^n \squareMatrix_{N\left(j\right)}\left(\coefficientsField\left[\GL_j\left(\finiteField\right)\right]\right)$ is the primitive central idempotent we are seeking.
	
	Tracing back the isomorphisms discussed in \Cref{subsec:overview-of-rep-theory-of-mn-fq} we see that $f_{\sigma}$ corresponds to the following element of the monoid algebra $\coefficientsField\left[\squareMatrix_n\left(\finiteField\right)\right]$:
	\begin{displaymath}
	e_{L_{n}\left(\sigma\right)} = \frac{\dim \sigma}{\sizeof{\GL_r\left(\finiteField\right)}} \sum_{\substack{V \subset \finiteField^n\\
	\dim V = r}} \sum_{g \in \GL_r\left(\finiteField\right)} \trace \sigma^{\vee}\left(g\right) \left[\phi_{V}^{-1} \circ g \circ \phi'_V\right] \eta_V,
	\end{displaymath}
	where $\phi'_V$ is any extension of $\phi_V$ to a linear map $\phi'_V \colon \finiteField^n \to \finiteField^r$. Expanding the definition of $\eta_V$, this becomes
	\begin{equation*}
		\begin{split}
			e_{L_{n}\left(\sigma\right)} =& \frac{\dim \sigma}{\sizeof{\GL_r\left(\finiteField\right)}} q^{-r\left(n-r\right)} \mu\left(r\right)^{-1} \sum_{\substack{V \subset \finiteField^n\\
					\dim V = r}} \sum_{g \in \GL_r\left(\finiteField\right)} \\
				& \times \sum_{\substack{S \colon \finiteField^n \to \finiteField^n\\
					S \text{ semi-idempotent}\\
					\ImageOf S \subset V}} \trace \sigma^{\vee}\left(g\right) \mu\left(\srank S\right) \left[\phi_{V}^{-1} \circ g \circ \phi_V \circ S\right].
		\end{split}
	\end{equation*}
	Changing variables $g \mapsto \phi_V^{-1} \circ g \circ \phi_V$ we have that
	\begin{equation*}
		\begin{split}
			e_{L_{n}\left(\sigma\right)} =& \frac{\dim \sigma}{\sizeof{\GL_r\left(\finiteField\right)}} q^{-r\left(n-r\right)} \mu\left(r\right)^{-1} \sum_{\substack{V \subset \finiteField^n\\
					\dim V = r}} \sum_{g \in \GL\left(V\right)} \sum_{\substack{S \colon \finiteField^n \to \finiteField^n\\
					S \text{ semi-idempotent}\\
					\ImageOf S \subset V}} \trace \sigma^{\vee}\left(g\right) \mu\left(\srank S\right) \left[g \circ S\right].
		\end{split}
	\end{equation*}
	
	Fix a linear subspace $V \subset \finiteField^n$ of dimension $r$ and consider the inner sum
	\begin{equation}\label{eq:inner-idempotent-sum}
		\sum_{g \in \GL\left(V\right)} \sum_{\substack{S \colon \finiteField^n \to \finiteField^n\\
				S \text{ semi-idempotent}\\
				\ImageOf S \subset V}} \trace \sigma^{\vee}\left(g\right) \mu\left(\srank S\right) \left[g \circ S\right].
	\end{equation}
	Let $A \colon \finiteField^n \to \finiteField^n$ be a linear map with $\ImageOf A \subset V$, and let $\delta_A$ be the indicator function of $A$.
	
	It follows that the coefficient of $\left[A\right]$ in \eqref{eq:inner-idempotent-sum} is given by
	\begin{equation}\label{eq:coefficient-of-A-in-e-L-n-sigma}
		\sum_{g \in \GL\left(V\right)} \sum_{\substack{S \colon \finiteField^n \to \finiteField^n\\
				S \text{ semi-idempotent}\\
			\ImageOf S \subset V}} \trace \sigma^{\vee}\left(g\right) \mu\left(\srank S\right) \delta_A\left(g \circ S\right).
	\end{equation}
	Notice that $S \colon \finiteField^n \to \finiteField^n$ with $\ImageOf S \subset V$ is semi-idempotent if and only if $S \restriction_V$ is semi-idempotent. Fix a decomposition $\finiteField^n = V \oplus U$. Given a semi-idempotent element $m \in \EndomorphismRing\left(V\right)$ and a linear map $T \colon U \to V$ let $j\left(m, T\right)$ be the linear map whose restriction to $V$ is $m$ and whose restriction to $U$ is $T$. Then the assignment $\left(m,T\right) \mapsto j\left(m, T\right)$ is a bijection between pairs $\left(m, T\right)$ as above and semi-idempotent linear maps $S \colon \finiteField^n \to \finiteField^n$ with $\ImageOf S \subset V$. Moreover, we have that $\srank j\left(m,T\right) = \srank m$. Thus \eqref{eq:coefficient-of-A-in-e-L-n-sigma} may be rewritten as
	\begin{equation}\label{eq:coefficient-of-A-in-e-L-n-sigma-with-j}
	\sum_{g \in \GL\left(V\right)} \sum_{\substack{m \colon V \to V\\
			m \text{ semi-idempotent}}} \sum_{T \colon U \to V} \trace \sigma^{\vee}\left(g\right) \mu\left(\srank m\right) \delta_A\left(g \circ j\left(m,T\right)\right).
	\end{equation}	
	Notice now that $g \circ j\left(m,T\right) = A$ if and only if $g \circ T = A \restriction_{U}$ and $g \circ m = A \restriction_{V}$. Thus we have that $T = g^{-1} \circ A\restriction_{U}$ is  determined by $A$ and $g$. It follows that \eqref{eq:coefficient-of-A-in-e-L-n-sigma-with-j} can be rewritten as
	\begin{equation}\label{eq:coefficient-of-A-in-e-L-n-sigma-with-j-over-V}
	\sum_{g \in \GL\left(V\right)} \sum_{\substack{m \colon V \to V\\
			m \text{ semi-idempotent}}} \trace \sigma^{\vee}\left(g\right) \mu\left(\srank m\right) \delta_{A \restriction_V}\left(g \circ m\right).
	\end{equation}		
	
	Identifying $\GL\left(V\right)$ with $\GL_r\left(\finiteField\right)$ and applying the regularized Godement--Jacquet functional equation (\Cref{thm:godement-jacquet-for-arbitrary-functions}) for $\GL_r\left(\finiteField\right)$ with $f = \delta_{A \restriction_V}$, we have that \eqref{eq:coefficient-of-A-in-e-L-n-sigma-with-j-over-V} equals 
	\begin{displaymath}
	\frac{\mu\left(r\right)}{\GaussSumScalar{\sigma^{\vee}}{\fieldCharacter}} \cdot q^{-\frac{r^2}{2}} \sum_{g \in \GL\left(V\right)} \fieldCharacter\left(\trace\left(A \restriction_{V} \circ g^{-1}\right)\right) \trace \sigma^{\vee}\left(g\right).
	\end{displaymath}
	Recalling the definition of $\GaussSumScalarFunction{\sigma^{\vee}}{\fieldCharacter}$, we conclude the following result.
	\begin{theorem}\label{thm:ln-sigma-idempotent-in-terms-of-gauss-sums}
		Let $\sigma$ be an irreducible representation of $\GL_r\left(\finiteField\right)$ for $0 \le r \le n$. Then the primitive central idempotent $e_{L_n\left(\sigma\right)} \in \coefficientsField\left[\squareMatrix_n\left(\finiteField\right)\right]$ corresponding to the simple module $L_n\left(\sigma\right)$ of $\coefficientsField\left[\squareMatrix_n\left(\finiteField\right)\right]$ is given by
		\begin{displaymath}
		e_{L_n\left(\sigma\right)} = \frac{\dim \sigma}{\sizeof{\GL_r\left(\finiteField\right)}} q^{-r\left(n-r\right)} \sum_{A \in \squareMatrix_n\left(\finiteField\right)} \sum_{\substack{V \subset \finiteField^n\\
				\dim V = r\\
				\ImageOf A \subset V}} \frac{\GaussSumScalarFunction{\sigma^{\vee}}{\fieldCharacter}\left(A \restriction_{V} \right)}{\GaussSumScalar{\sigma^{\vee}}{\fieldCharacter}} \left[A\right].
				\end{displaymath}
	\end{theorem}
	
	By \eqref{eq:expression-of-pi-natural-with-gauss-sums} and the definition of $L_n\left(\sigma\right)^{\natural}$, see \Cref{subsec:functional-equation-for-monoid-algebras}, we have that
	\begin{displaymath}
		\trace L_n\left(\sigma^{\vee}\right)^{\natural}\left(A\right) = \sum_{\substack{V \subset \finiteField^n\\
				\dim V = r\\
				\ImageOf A \subset V}} \frac{\GaussSumScalarFunction{\sigma^{\vee}}{\fieldCharacter}\left(A \restriction_V\right)}{\GaussSumScalar{\sigma^{\vee}}{\fieldCharacter}}.
	\end{displaymath}
	Note that
	\begin{equation}\label{eq:dimension-of-l-n-sigma}
		\dim L_{n}\left(\sigma\right) = \dim \sigma \cdot \frac{\sizeof{\GL_n\left(\finiteField\right)}}{q^{r \left(n-r\right)} \sizeof{\GL_r\left(\finiteField\right)} \sizeof{\GL_{n-r}\left(\finiteField\right)}}.
	\end{equation}
	From these observations we have the following formula for $e_{L_n\left(\sigma\right)}$, which is similar to the formula for $e_{\sigma}$ (see \eqref{eq:idempotent-for-irreducible-representation-of-group}).
	\begin{corollary}
		Keep the notation as in \Cref{thm:ln-sigma-idempotent-in-terms-of-gauss-sums}. Then we have
		\begin{displaymath}
			e_{L_n\left(\sigma\right)} = \frac{\dim L_n\left(\sigma\right)}{\sizeof{\GL_n\left(\finiteField\right)}} \sizeof{\GL_{n-r}\left(\finiteField\right)} \sum_{A \in \squareMatrix_n\left(\finiteField\right)} \trace L_n\left(\sigma^{\vee}\right)^{\natural}\left(A\right) \left[A\right].
		\end{displaymath}
	\end{corollary}
	
	From these observations combined with \Cref{thm:godement-jacquet-for-monoid-algebras} we obtain the following corollary.
	\begin{corollary}\label{cor:trace-godement-jacquet-for-m-n-fq}
		Let $\sigma$ be an irreducible representation of $\GL_r\left(\finiteField\right)$ for $0 \le r \le n$. Define $\mathcal{X}_{L_n\left(\sigma^{\vee}\right)} \colon \squareMatrix_n\left(\finiteField\right) \to \coefficientsField$ by the equality
		\begin{displaymath}
			e_{L_n\left(\sigma\right)} = \sum_{X \in \squareMatrix_n\left(\finiteField\right)} \mathcal{X}_{L_n\left(\sigma^{\vee}\right)}\left(X\right) \left[X\right].
		\end{displaymath}
		Then for every function $f \in \Schwartz\left(\squareMatrix_n\left(\finiteField\right)\right)$ we have the equality
		\begin{displaymath}
			\sum_{A \in \squareMatrix_n\left(\finiteField\right)} f\left(A\right) \mathcal{X}_{L_n\left(\sigma^{\vee}\right)}\left(A\right) = q^{-\frac{\left(n-r\right)^2}{2}}  \frac{\dim L_n\left(\sigma\right)}{\GaussSumScalar{\sigma^{\vee}}{\fieldCharacter}} \frac{\sizeof{\GL_{n-r}\left(\finiteField\right)}}{\sizeof{\GL_n\left(\finiteField\right)}} \sum_{A \in \squareMatrix_n\left(\finiteField\right)} \fourierTransform{\fieldCharacter}f\left(A\right) \trace L_{n}\left(\sigma\right)\left(A\right).
		\end{displaymath}
	\end{corollary}
	
	\subsubsection{Rewriting using parabolic induction}\label{subsubsec:rewrite-using-parabolic-induction}
	Suppose that $n = n_1 + n_2$ and that $f_i \colon \squareMatrix_{n_i}\left(\finiteField\right) \to \coefficientsField$ is a $\GL_{n_i}\left(\finiteField\right)$-invariant function for $i=1,2$. We may define a $\GL_n\left(\finiteField\right)$-invariant function $f_1 \odot f_2 \colon \squareMatrix_n\left(\finiteField\right) \to \coefficientsField$ by the formula
	\begin{displaymath}
	\left(f_1 \odot f_2\right)\left(A\right) = \sum_{\substack{V \subset \finiteField^n\\
	\dim V = n_1 \\
	A V \subset V}} f_1\left(A \restriction_V \right) f_2\left(A_{\finiteField^n \slash V} \right).
	\end{displaymath}
	Here $A_{\finiteField^n \slash V} \colon \finiteField^n \slash V \to \finiteField^n \slash V$ is the linear map given by $A_{\finiteField^n \slash V}\left(x + V\right) = Ax + V$ for $x \in \finiteField^n$. This operation is analogous to the usual parabolic induction operation defined for class functions of finite general linear groups.
	
	Let $\delta_{0_{n-r}} \colon \squareMatrix_{n-r}\left(\finiteField\right) \to \coefficientsField$ be the indicator function of $0_{n-r}$. Let $\sigma$ be an irreducible representation of $\GL_r\left(\finiteField\right)$ for $0 \le r \le n$. Recall that $\GaussSumScalarFunction{\sigma^{\vee}}{\fieldCharacter}$ is a $\GL_r\left(\finiteField\right)$-invariant function. Thus \Cref{thm:ln-sigma-idempotent-in-terms-of-gauss-sums} can be restated as
	\begin{displaymath}
	e_{L_n\left(\sigma\right)} = \frac{\dim \sigma}{\sizeof{\GL_r\left(\finiteField\right)}} q^{-r\left(n-r\right)} \GaussSumScalar{\sigma^{\vee}}{\fieldCharacter}^{-1} \sum_{A \in \squareMatrix_n\left(\finiteField\right)} \left(\GaussSumScalarFunction{\sigma^{\vee}}{\fieldCharacter} \odot \delta_{0_{n-r}}\right) \left(A\right) \left[A\right].
	\end{displaymath}
	Combining with \eqref{eq:dimension-of-l-n-sigma}, this allows us to rewrite $e_{L_n\left(\sigma\right)}$ in the following form.
	\begin{theorem}\label{thm:ln-sigma-idempotent-in-terms-of-parabolic-induction-and-gauss-sums}
		Let $\sigma$ be an irreducible representation of $\GL_r\left(\finiteField\right)$ for $0 \le r \le n$. Then the primitive central idempotent $e_{L_n\left(\sigma\right)} \in \coefficientsField\left[\squareMatrix_n\left(\finiteField\right)\right]$ corresponding to the simple module $L_n\left(\sigma\right)$ of $\coefficientsField\left[\squareMatrix_n\left(\finiteField\right)\right]$ is given by
		\begin{displaymath}
		e_{L_n\left(\sigma\right)} = \frac{\dim L_n\left(\sigma\right)}{\sizeof{\GL_n\left(\finiteField\right)}} \sizeof{\GL_{n-r}\left(\finiteField\right)} \cdot \GaussSumScalar{\sigma^{\vee}}{\fieldCharacter}^{-1} \sum_{A \in \squareMatrix_n\left(\finiteField\right)} \left(\GaussSumScalarFunction{\sigma^{\vee}}{\fieldCharacter} \odot \delta_{0_{n-r}}\right) \left(A\right) \left[A\right].
		\end{displaymath}
	\end{theorem}	
	
	\begin{remark}
		Let $\chi_{\sigma} \colon \squareMatrix_r\left(\finiteField\right) \to \coefficientsField$ be the character of $\sigma$, extended to zero outside of $\GL_r\left(\finiteField\right)$, i.e.,
		\begin{displaymath}
			\chi_{\sigma}\left(A\right) = \begin{dcases}
				\trace \sigma\left(A\right) & \text{if }A \in \GL_r\left(\finiteField\right),\\
				0 & \text{otherwise}.
			\end{dcases}
		\end{displaymath}
		By~\cite[Proposition 5.1]{HarmanSnowdenZelingher2025} and the definition of parabolic induction we have that for every $A \in \squareMatrix_n\left(\finiteField\right)$,
		\begin{displaymath}
			\trace L_n\left(\sigma\right)\left(A\right) = \left(\chi_{\sigma} \odot 1_{\squareMatrix_{n-r}}\right)\left(A\right),
		\end{displaymath}
		where $1_{\squareMatrix_{n-r}}$ is the constant function $1$ on $\squareMatrix_{n-r}\left(\finiteField\right)$. By~\cite{GorodetskyZelingher2026} we have that if $f_i$ is a $\GL_{n_i}\left(\finiteField\right)$-invariant function for $i=1,2$, then
		\begin{displaymath}
			\fourierTransform{\fieldCharacter}\left(f_1 \odot f_2\right) = \fourierTransform{\fieldCharacter}\left(f_1\right) \odot \fourierTransform{\fieldCharacter}\left(f_2\right).
		\end{displaymath}
		Since $\fourierTransform{\fieldCharacter}\left(1_{\squareMatrix_{n-r}}\right) = q^{\frac{\left(n-r\right)^2}{2}} \delta_{0_{n-r}}$ and since $\fourierTransform{\fieldCharacter}\left(\chi_{\sigma}\right) = \GaussSumScalarFunction{\sigma^{\vee}}{\fieldCharacter}$ we have
		\begin{equation}\label{eq:formula-for-coefficients-as-fourier-transform}
			\GaussSumScalarFunction{\sigma^{\vee}}{\fieldCharacter} \odot \delta_{0_{n-r}} = q^{-\frac{\left(n-r\right)^2}{2}} \fourierTransform{\fieldCharacter}\left(\trace L_n\left(\sigma\right)\right).
		\end{equation}
		\Cref{cor:trace-godement-jacquet-for-m-n-fq} can be deduced from \eqref{eq:formula-for-coefficients-as-fourier-transform} using the Parseval identity~\cite[Below Equation (2.2)]{KurinczukMatringeSecherre2026}.
	\end{remark}
	
	\subsubsection{Explicit formulas for Gauss sums}\label{subsubsec:explicit-formulas-for-gauss-sums}
	
	We finish this section with some explicit formulas for the Gauss sums appearing in \Cref{thm:ln-sigma-idempotent-in-terms-of-gauss-sums, thm:ln-sigma-idempotent-in-terms-of-parabolic-induction-and-gauss-sums}.
	
	First, from \eqref{eq:expression-of-pi-natural-with-gauss-sums} and from the fact that $\left(\sigma^{\vee}\right)^{\natural}\left(A\right) = \sigma^{\vee}\left(A\right)$ for $A \in \GL_r\left(\finiteField\right)$, we have that if $\sigma$ is an irreducible representation of $\GL_r\left(\finiteField\right)$ and if $A \in \GL_r\left(\finiteField\right)$ then
	\begin{displaymath}
	\frac{\GaussSumScalarFunction{\sigma^{\vee}}{\fieldCharacter}\left(A\right)}{\GaussSumScalar{\sigma^{\vee}}{\fieldCharacter}} = \trace \sigma^{\vee}\left(A\right).
	\end{displaymath}
	Thus the assignment $\squareMatrix_r\left(\finiteField\right) \to \coefficientsField$ given by $A \mapsto \GaussSumScalar{\sigma^{\vee}}{\fieldCharacter}^{-1} \GaussSumScalarFunction{\sigma^{\vee}}{\fieldCharacter}\left(A\right)$ can be thought of as ``meromorphic continuation'' of the class function $\trace \sigma^{\vee}\left(g\right)$ of $\GL_r\left(\finiteField\right)$.
	
	Next, we have by \Cref{prop:strict-tensor-co-rank-and-vanishing} that if $\sigma$ is pure (see below for the definition) then $\GaussSumScalarFunction{\sigma^{\vee}}{\fieldCharacter}$ vanishes at singular elements.
	
	Finally, we recall our most general result from~\cite{GorodetskyZelingher2026}. See also~\cite{Kawanaka1985, Lusztig1992, Letellier2005} for related results. To state it, we first recall some terminology. Recall that an irreducible representation of $\GL_r\left(\finiteField\right)$ is called \emph{unipotent} if it is a subrepresentation of the parabolically induced representation $1^{\odot r} = 1 \odot \dots \odot 1$, where $1$ is the trivial character of $\GL_1\left(\finiteField\right)$. Irreducible unipotent representations of $\GL_r\left(\finiteField\right)$ are in bijection with partitions of $r$. Given a partition $\lambda$ of $r$ we let $1^{\lambda}$ denote the unipotent irreducible representation of $\GL_r\left(\finiteField\right)$ corresponding to $\lambda$ with the convention that $1^{\left(r\right)}$ is the Steinberg representation and $1^{(1^r)}$ is the trivial representation. If $\tau$ is an irreducible representation of $\GL_r\left(\finiteField\right)$, we say that $\tau$ is \emph{pure} if $\tau$ is not a subrepresentation of a parabolically induced representation of the form $1 \odot \tau'$ where $\tau'$ is a representation of $\GL_{r-1}\left(\finiteField\right)$. Every irreducible representation of $\GL_r\left(\finiteField\right)$ can be realized as a parabolically induced representation of the form $\tau' \odot 1^{\lambda}$ where $\lambda \vdash r'$ and $\tau'$ is a pure irreducible representation of $\GL_{r-r'}\left(\finiteField\right)$ for $0 \le r' \le r$.
	
	For a positive integer $m$ and $\xi \in \finiteField$ let $J_{(m)}\left(\xi\right)$ be the Jordan block of size $m$ with eigenvalue $\xi$, i.e.,
	\begin{displaymath}
	J_{(m)}\left(\xi\right) = \begin{pmatrix}
		\xi & 1\\
		& \xi & \ddots\\
		& & \ddots & 1\\
		& & & \xi
	\end{pmatrix} \in \squareMatrix_m\left(\finiteField\right).
	\end{displaymath}
	If $\mu = \left(m_1,\dots, m_{\ell}\right)$ is a partition with $m_1 \ge \dots \ge m_{\ell} > 0$ where $\ell \ge 0$, then we set $J_{\mu}\left(\xi\right)$ to be the Jordan matrix with eigenvalue $\xi$ corresponding to $\mu$, i.e.,
	\begin{displaymath}
	 J_{\mu}\left(\xi\right) = \diag\left(J_{(m_1)}\left(\xi\right),\dots,J_{(m_{\ell})}\left(\xi\right)\right).
	\end{displaymath}
	We let $a_{\mu}\left(q\right)$ denote the size of the centralizer of $J_{\mu}\left(\xi\right)$ in $\GL_{\sizeof{\mu}}\left(\finiteField\right)$. It is independent of $\xi$.
	
	Given $r \ge 0$, consider the modified Kostka--Foulkes matrix $\left(\tilde{\kostkaFoulkes}_{\lambda,\mu}\left(q\right)\right)_{\lambda,\mu \vdash r}$, where $\tilde{\kostkaFoulkes}_{\lambda,\mu}\left(q\right)$ is the modified Kostka--Foulkes polynomial, also given by $\tilde{\kostkaFoulkes}_{\lambda,\mu}\left(q\right) = \trace 1^{\lambda^{\top}}\left(J_{\mu}\left(1\right)\right)$, where $\lambda^{\top}$ is the conjugate of $\lambda$. Let $\left(\coinvKostkaFoulkes_{\lambda, \mu}\left(q\right)\right)_{\lambda, \mu \vdash r}$ be the inverse of the modified Kostka--Foulkes matrix.
	
	For partitions $\lambda$, $\nu$ and $\rho$ such that $\lambda \vdash \sizeof{\nu} + \sizeof{\rho}$, let $c_{\rho, \nu}^{\lambda}$ be the Littlewood--Richardson coefficients, defined by the rule
	\begin{displaymath}
	s_{\rho} \cdot s_{\nu} = \sum_{\lambda \vdash \sizeof{\rho} + \sizeof{\nu}} c_{\rho, \nu}^{\lambda} s_{\lambda},
	\end{displaymath}
	where $s_{\lambda}$ is the Schur function corresponding to $\lambda$.
	
	We are now ready to state the main result of~\cite{GorodetskyZelingher2026}.
	
	\begin{theorem}
		Let $\sigma_0$ be a pure irreducible representation of $\GL_s\left(\finiteField\right)$ for $0 \le s \le r$ and let $\sigma = \sigma_0 \odot 1^{\lambda}$ for $\lambda \vdash r-s$. Let $A = \diag\left(J_{\mu}\left(0\right), g\right)$ where $\mu$ is a partition of size $\le r$ and $g \in \GL_{r - \sizeof{\mu}}\left(\finiteField\right)$. Then $\GaussSumScalarFunction{\sigma}{\fieldCharacter}\left(A\right)$ vanishes unless $\sizeof{\lambda} \ge \sizeof{\mu}$, in which case
		\begin{displaymath}
		\GaussSumScalarFunction{\sigma}{\fieldCharacter}\left(A\right) = \left(-1\right)^{\sizeof{\mu}} a_{\mu}\left(q\right) \GaussSumScalar{\sigma}{\fieldCharacter} \sum_{\nu \vdash \sizeof{\lambda} - \sizeof{\mu}} \sum_{\rho \vdash \sizeof{\mu}} c_{\nu, \rho}^{\lambda} \coinvKostkaFoulkes_{\mu, \rho}\left(q\right) \trace\left(\sigma_0 \odot 1^{\nu}\right)\left(g\right).
		\end{displaymath}
	\end{theorem}
	Using this result and \Cref{thm:ln-sigma-idempotent-in-terms-of-parabolic-induction-and-gauss-sums} one can write down an explicit formula for the primitive central idempotent $e_{L_n\left(\sigma\right)}$.

	\appendix	
	
	\section{Relation to strict tensor co-rank}\label{sec:representations-with-strict-tensor-corank}
	In this appendix we give a criterion for $\pi^{\natural}\left(A\right)$ to be non-zero for singular $A \in \squareMatrix_n\left(\finiteField\right)$. For $0 \le k \le n$ set $A_{k} = \left(\begin{smallmatrix}
		\IdentityMatrix{n-k}\\
		& 0_k
	\end{smallmatrix}\right)$. Recall that any singular matrix $A$ can be written in the form $A = g_1 A_j g_2$, where $1 \le j \le n$ and $g_1, g_2 \in \GL_n\left(\finiteField\right)$. By \eqref{eq:equivariance-of-pi-natural}, for such $A$ we have that  $\pi^{\natural}\left(A\right) = 0$ if and only if $\pi^{\natural}\left(A_j\right) = 0$.
	
	For $0 \le k \le n$, define
	\begin{displaymath}
	H^{+}_{k} = \left\{\begin{pmatrix}
		\IdentityMatrix{n-k} & X\\
		& g
	\end{pmatrix} \mid X \in \Mat{\left(n-k\right)}{k}\left(\finiteField\right), g \in \GL_k\left(\finiteField\right)\right\}.
	\end{displaymath}
	Similarly, define
	\begin{displaymath}
		H^{-}_{k} = \left\{\begin{pmatrix}
			\IdentityMatrix{n-k} \\
			X & g
		\end{pmatrix} \mid X \in \Mat{k}{\left(n-k\right)}\left(\finiteField\right), g \in \GL_k\left(\finiteField\right)\right\}.
	\end{displaymath}

	If $\pi$ is an irreducible representation of $\GL_n\left(\finiteField\right)$, we say that $\pi$ has \emph{strict tensor co-rank $k$} if $0 \le k \le n$ is the largest integer such that $\pi$ admits a non-zero $H^+_k$-fixed vector. 
	
	\begin{proposition}\label{prop:strict-tensor-co-rank-and-vanishing}
		Suppose that $\pi$ has strict tensor co-rank $r$. Then $\pi^{\natural}\left(A\right) = 0$ for any $A \in \squareMatrix_n\left(\finiteField\right)$ with $\rank A < n-r$.
	\end{proposition}
	\begin{proof}
		Suppose that $A$ has rank $k$, where $0 \le k < n - r$. As explained above, it suffices to prove that $\pi^{\natural}\left(A_{n-k}\right) = 0$. Notice that
		\begin{displaymath}
			\pi\begin{pmatrix}
				\IdentityMatrix{k} & X\\
				& g
			\end{pmatrix} \circ \pi^{\natural}\begin{pmatrix}
			\IdentityMatrix{k}\\
			& 0_{n-k}
			\end{pmatrix} = \pi^{\natural}\begin{pmatrix}
			\IdentityMatrix{k}\\
			& 0_{n-k}
			\end{pmatrix},
		\end{displaymath}
		for every $X \in \Mat{k}{\left(n-k\right)}\left(\finiteField\right)$ and $g \in \GL_{n-k}\left(\finiteField\right)$. 
		
		Thus if $\pi^{\natural}\left(A\right) \ne 0$ we have that $\pi^{\natural}\left(A_{n-k}\right) \ne 0$ which implies that $\pi$ has strict tensor co-rank $\ge n-k>r$, which is a contradiction.	
	\end{proof}
	
	Assume from now on that $\coefficientsField$ has characteristic not dividing $\sizeof{\GL_n\left(\finiteField\right)}$.
	
	For an irreducible representation $\pi$ of $\GL_n\left(\finiteField\right)$, let
	\begin{displaymath}
	\pnHJacquetModule{\pi}{k} = \left\{v \in \pi \mid \pi\left(h\right) v = v \, \text{ for all } h \in H_k^{\pm} \right\}.
	\end{displaymath}
	Using the fact that $\pi^{\vee}$ is isomorphic to $\pi^{\iota}$, one can show that $\dim \pHJacquetModule{\pi}{k} = \dim \nHJacquetModule{\pi}{k}$.
	
	The Levi subgroup
	\begin{displaymath}
	L_{\left(n-k, k\right)} = \left\{\diag\left(g_1, g_2\right) \mid g_1 \in \GL_{n-k}\left(\finiteField\right), g_2 \in \GL_{k}\left(\finiteField\right) \right\} \cong \GL_{n-k}\left(\finiteField\right) \times \GL_k\left(\finiteField\right)
	\end{displaymath}
	preserves the spaces $\pHJacquetModule{\pi}{k}$ and $\nHJacquetModule{\pi}{k}$. Moreover, the second component of $L_{\left(n-k,k\right)}$ acts trivially on both spaces.
	
	Let $\ProjectionOperator_{\pnHJacquetModule{\pi}{k}} \colon \pi \to \pnHJacquetModule{\pi}{k}$ be the projection operator. It is given by the formula
	\begin{displaymath}
	\ProjectionOperator_{\pnHJacquetModule{\pi}{k}} = \frac{1}{\sizeof{H_k^{\pm}}} \sum_{h \in H_k^{\pm}} \pi\left(h\right).
	\end{displaymath}

	We may now state our result about the degenerate Gauss sum corresponding to $A_k$.
	\begin{theorem}\label{thm:gauss-sum-in-terms-of-projection-operator}
		For $0 \le k \le n$ let \begin{displaymath}
			\GaussSum{\pi}{\fieldCharacter}_k = q^{-\frac{n^2}{2}} \sum_{g \in \GL_n\left(\finiteField\right)} \fieldCharacter\left(\trace\left(\begin{pmatrix}
				I_{n-k}\\
				& 0_k
			\end{pmatrix} g^{-1}\right)\right) \pi\left(g\right).
		\end{displaymath}
		Suppose that $\pi$ is of strict tensor co-rank $\ge k$. Then the operator
			\begin{displaymath}
				T = \ProjectionOperator_{\nHJacquetModule{\pi}{k}} \restriction_{\pHJacquetModule{\pi}{k}} \colon \pHJacquetModule{\pi}{k} \to \nHJacquetModule{\pi}{k}
			\end{displaymath}
			is an isomorphism, and $\GaussSum{\pi}{\fieldCharacter}_k$ is given by
			\begin{displaymath}
			\GaussSum{\pi}{\fieldCharacter}_k = q^{-\frac{k^2}{2}} \sizeof{\GL_{k}\left(\finiteField\right)} \GaussSumScalar{\sigma}{\fieldCharacter} T^{-1} \circ \ProjectionOperator_{\nHJacquetModule{\pi}{k}},
			\end{displaymath}
			where $\sigma$ is any irreducible representation of $\GL_{n-k}\left(\finiteField\right)$ such that $\sigma \otimes 1_{\GL_k}$ is a subrepresentation of $\pHJacquetModule{\pi}{k}$.
	\end{theorem}
	\begin{proof}
		Since $h^+ \left(\begin{smallmatrix}
			\IdentityMatrix{n-k}\\
			& 0_k
		\end{smallmatrix}\right) h^- = \left(\begin{smallmatrix}
			\IdentityMatrix{n-k}\\
			& 0_k
		\end{smallmatrix}\right)$ for any $h^+ \in H^{+}_k$ and $h^- \in H^{-}_k$, we have that
		\begin{displaymath}
		\GaussSum{\pi}{\fieldCharacter}_k = \ProjectionOperator_{\pHJacquetModule{\pi}{k}} \circ \GaussSum{\pi}{\fieldCharacter}_k \circ \ProjectionOperator_{\nHJacquetModule{\pi}{k}}.
		\end{displaymath}
		
		Consider the following element of the group algebra $\coefficientsField\left[\GL_n\left(\finiteField\right)\right]$:
		\begin{displaymath}
		e_{\fieldCharacter, k}^{\vee} = q^{-\frac{n^2}{2}} \sum_{g \in \GL_n\left(\finiteField\right)} \fieldCharacter\left(\trace \left(\begin{pmatrix}
			\IdentityMatrix{n-k}\\
			& 0_k
		\end{pmatrix}g^{-1}\right)\right)\left[g\right].
		\end{displaymath}
		Also consider the element
		\begin{displaymath}
		e_{H^+_k} = \frac{1}{\sizeof{H_k^{+}}}\sum_{h^+ \in H^{+}_k} \left[h^{+}\right].
		\end{displaymath}
		Then the product $e_{\fieldCharacter, k}^{\vee} \cdot e_{H^+_k}$ equals
		\begin{displaymath}
		q^{-\frac{n^2}{2}} \frac{1}{\sizeof{H_k^{+}}} \sum_{\substack{g \in \GL_n\left(\finiteField\right)\\
				h^+ \in H^{+}_k}} \fieldCharacter\left(\trace \left(\begin{pmatrix}
			\IdentityMatrix{n-k}\\
			& 0_k
		\end{pmatrix} h^{+} g^{-1}\right)\right)\left[g\right],
		\end{displaymath}
		which equals
		\begin{displaymath}
		q^{-\frac{n^2}{2}-\left(n-k\right)k} \sum_{g \in \GL_n\left(\finiteField\right)} \sum_{X \in \Mat{\left(n-k\right)}{k}\left(\finiteField\right)} \fieldCharacter\left(\trace \left(\begin{pmatrix}
			\IdentityMatrix{n-k} & X\\
			& 0_{k}
		\end{pmatrix} g^{-1} \right)\right) \left[g\right].
		\end{displaymath}
		Fix $g \in \GL_n\left(\finiteField\right)$ and write
		\begin{displaymath}
		g^{-1} = \begin{pmatrix}
			g_{11} & g_{12}\\
			g_{21} & g_{22}
		\end{pmatrix},
		\end{displaymath}
		where $g_{11} \in \squareMatrix_{n-k}\left(\finiteField\right)$, $g_{12} \in \Mat{\left(n-k\right)}{k}\left(\finiteField\right)$, $g_{21} \in \Mat{k}{\left(n-k\right)}\left(\finiteField\right)$, and $g_{22} \in \squareMatrix_{k}\left(\finiteField\right)$. Then the inner sum over $X$ equals
		\begin{displaymath}
		\sum_{X \in \Mat{\left(n-k\right)}{k}\left(\finiteField\right)} \fieldCharacter\left(\trace g_{11} + \trace\left(Xg_{21}\right)\right).
		\end{displaymath}
		The last sum is zero unless $g_{21} = 0$. Writing $g = \left(\begin{smallmatrix}
			h_{11} & h_{12}\\
			 & h_{22}
		\end{smallmatrix}\right)$ with $h_{11} \in \GL_{n-k}\left(\finiteField\right)$, $h_{12} \in \Mat{\left(n-k\right)}{k}\left(\finiteField\right)$, and $h_{22} \in \GL_k\left(\finiteField\right)$, we have that $e_{\fieldCharacter, k}^{\vee} \cdot e_{H^+_k}$ equals
		\begin{displaymath}
		q^{-\frac{n^2}{2}} \sum_{\substack{h_{11} \in \GL_{n-k}\left(\finiteField\right)\\
				h_{12} \in \Mat{\left(n-k\right)}{k}\left(\finiteField\right)\\
				h_{22} \in \GL_{k}\left(\finiteField\right) }} \fieldCharacter\left(\trace h_{11}^{-1}\right) \left[\begin{pmatrix}
			h_{11} & h_{12}\\
			& h_{22}
		\end{pmatrix}\right],
		\end{displaymath}
		which is
		\begin{displaymath}
		q^{-\frac{n^2}{2}} \sizeof{H^+_k} \left(\sum_{h_{11} \in \GL_{n-k}\left(\finiteField\right)} \fieldCharacter\left(\trace h_{11}^{-1}\right) \left[\begin{pmatrix}
			h_{11}\\
			& \IdentityMatrix{k}
		\end{pmatrix}\right]\right) \cdot e_{H_k^+}.
		\end{displaymath}
		Thus $\GaussSum{\pi}{\fieldCharacter}_k \circ \ProjectionOperator_{\pHJacquetModule{\pi}{k}} = \pi\left(e_{\fieldCharacter, k}^{\vee} \cdot e_{H^+_k}\right)$ is given by
		\begin{displaymath}
		\GaussSum{\pi}{\fieldCharacter}_k \circ \ProjectionOperator_{\pHJacquetModule{\pi}{k}} = q^{-\frac{n^2}{2}} \sizeof{H^+_k} \sum_{h_{11} \in \GL_{n-k}\left(\finiteField\right)} \fieldCharacter\left(\trace h_{11}^{-1}\right) \pi\begin{pmatrix}
			h_{11}\\
			& \IdentityMatrix{k}
		\end{pmatrix} \circ \ProjectionOperator_{\pHJacquetModule{\pi}{k}}.
		\end{displaymath}
		Let us write
		\begin{displaymath}
		\pHJacquetModule{\pi}{k} = \bigoplus_{j=1}^{\ell} \sigma_j \otimes 1_{\GL_k\left(\finiteField\right)},
		\end{displaymath}
		where $\sigma_j$ are irreducible representations. Then $\sigma_1,\dots,\sigma_{\ell}$ all have the same cuspidal support. We have that
		\begin{displaymath}
		\sum_{h_{11} \in \GL_{n-k}\left(\finiteField\right)} \fieldCharacter\left(\trace h_{11}^{-1}\right) \pi\begin{pmatrix}
			h_{11}\\
			& \IdentityMatrix{k}
		\end{pmatrix}\restriction_{\sigma_j} = \sum_{h_{11} \in \GL_{n-k}\left(\finiteField\right)} \fieldCharacter\left(\trace h_{11}^{-1}\right) \sigma_j\left(h_{11}\right) = q^{\frac{\left(n-k\right)^2}{2}} \GaussSumScalar{\sigma_j}{\fieldCharacter} \cdot \idmap_{\sigma_j}.
		\end{displaymath}
		
		Since $\sigma_1$, $\dots$, $\sigma_{\ell}$ all have the same cuspidal support, we have from~\cite[Corollary 4.5]{KurinczukMatringeSecherre2026} that
		\begin{displaymath}
		\GaussSumScalar{\sigma_1}{\fieldCharacter} = \dots = \GaussSumScalar{\sigma_{\ell}}{\fieldCharacter},
		\end{displaymath}
		and thus
		\begin{displaymath}
		\GaussSum{\pi}{\fieldCharacter}_k \circ \ProjectionOperator_{\pHJacquetModule{\pi}{k}} = q^{-\frac{k^2}{2}} \sizeof{\GL_k\left(\finiteField\right)} \GaussSumScalar{\sigma}{\fieldCharacter} \ProjectionOperator_{\pHJacquetModule{\pi}{k}}.
		\end{displaymath}
		
		To summarize, we have that $\GaussSum{\pi}{\fieldCharacter}_k \circ T \circ \ProjectionOperator_{\pHJacquetModule{\pi}{k}} = \GaussSum{\pi}{\fieldCharacter}_k \circ \ProjectionOperator_{\pHJacquetModule{\pi}{k}}$ acts on $\pHJacquetModule{\pi}{k}$ by the non-zero scalar $q^{-\frac{k^2}{2}} \sizeof{\GL_k\left(\finiteField\right)} \GaussSumScalar{\sigma}{\fieldCharacter} \in \multiplicativegroup{\coefficientsField}$. Since $\pHJacquetModule{\pi}{k}$ and $\nHJacquetModule{\pi}{k}$ have the same dimension, it follows that $T$ is an isomorphism.
		
		Since the restriction of  $\ProjectionOperator_{\pHJacquetModule{\pi}{k}}$ to $\pHJacquetModule{\pi}{k}$ is the identity map we have that
		\begin{displaymath}
			\GaussSum{\pi}{\fieldCharacter}_k \restriction_{\nHJacquetModule{\pi}{k}} \circ T = q^{-\frac{k^2}{2}} \sizeof{\GL_k\left(\finiteField\right)} \GaussSumScalar{\sigma}{\fieldCharacter} \idmap_{\pHJacquetModule{\pi}{k}}.
		\end{displaymath}
		Combining this with $\GaussSum{\pi}{\fieldCharacter}_k = \GaussSum{\pi}{\fieldCharacter}_k \circ \ProjectionOperator_{\nHJacquetModule{\pi}{k}}$, we get the result.
	\end{proof}
	
	\begin{corollary}
		 For $0 \le k \le n$ let $\GaussSumScalar{\pi}{\fieldCharacter}_k = \GaussSumScalarFunction{\pi}{\fieldCharacter}\left(\begin{smallmatrix}
		 	\IdentityMatrix{n-k}\\
		 	& 0_{k}
		 \end{smallmatrix}\right)$.
		Then the Gauss sum $\GaussSumScalar{\pi}{\fieldCharacter}_k$ is zero unless $\pi$ is of strict tensor co-rank $\ge k$, in which case we have
		\begin{equation*}
			\GaussSumScalar{\pi}{\fieldCharacter}_k = q^{-\frac{k^2}{2}} \sizeof{\GL_k\left(\finiteField\right)} \dim \pHJacquetModule{\pi}{k} \cdot \GaussSumScalar{\sigma}{\fieldCharacter}.
		\end{equation*}
	\end{corollary}
	\begin{proof}
		From \Cref{prop:strict-tensor-co-rank-and-vanishing} combined with \eqref{eq:expression-of-pi-natural-with-gauss-sums} we have that $\GaussSumScalar{\pi}{\fieldCharacter}_k$ vanishes if $\pi$ is of strict tensor co-rank $< k$.
		
		Assume that $\pi$ has strict tensor co-rank $\ge k$. We have that
		\begin{displaymath}
			\GaussSumScalar{\pi}{\fieldCharacter}_k = \trace\left(\GaussSum{\pi}{\fieldCharacter}_k\right) = q^{-\frac{k^2}{2}} \sizeof{\GL_{k}\left(\finiteField\right)} \GaussSumScalar{\sigma}{\fieldCharacter} \trace\left(T^{-1} \circ \ProjectionOperator_{\nHJacquetModule{\pi}{k}}\right).
		\end{displaymath}
		We have that $T^{-1}$ has image in $\pHJacquetModule{\pi}{k}$ and that the restriction of $T^{-1} \circ \ProjectionOperator_{\nHJacquetModule{\pi}{k}}$ to $\pHJacquetModule{\pi}{k}$ is the identity map. Thus $\trace\left(T^{-1} \circ \ProjectionOperator_{\nHJacquetModule{\pi}{k}}\right) = \dim \pHJacquetModule{\pi}{k}$.
	\end{proof}

\bibliographystyle{abbrv}
\bibliography{gj-references}
\end{document}